\documentclass[10pt, a4paper]{amsart}
\pdfoutput=1

\usepackage[T1]{fontenc}
\usepackage{lmodern}
\usepackage[spacing,kerning, tracking]{microtype}

\usepackage[style=numams, sorting=nyt, isbn=false,  backref=true, natbib=true]{biblatex}
\usepackage{amsmath,amssymb, amscd, amsthm}
\usepackage{enumitem}
\usepackage{verbatim}

\usepackage[pdfa]{hyperref}

\newcommand{\ichselber}{Frieder Ladisch}
\newcommand{\meineaddresse}{Universit\"{a}t Rostock,
                        Institut f\"{u}r Mathematik,
                        Ulmenstr.~69, Haus~3,
                        18057 Rostock,
                        Germany}
\newcommand{\meineemail}{frieder.ladisch@uni-rostock.de}

\newcommand{\titellang}{Character correspondences
                        above fully ramified sections
                        and Schur indices}
\newcommand{\titellangsep}{Character correspondences \\ above fully ramified sections \\
                            and Schur indices}
\newcommand{\titelkurz}{Fully ramified sections}

\newcommand{\komma}{, }
\newcommand{\stichwoerter}{character theory of finite groups\komma
           fully ramified characters\komma
           Glauberman-Isaacs correspondence\komma
           Schur indices}

\newcommand{\mscwert}{Primary 20C15}

\hypersetup{pdfauthor=\ichselber,
            pdftitle={\titellang},
            pdfkeywords={\stichwoerter ,
                        2010 Mathematics subject classification: \mscwert}
            }

\usepackage{mathbef}

\newcommand{\llangle}{\langle\mspace{-4mu}\langle}\newcommand{\rrangle}{\rangle\mspace{-4mu}\rangle}
\newcommand{\skp}[1]{\llangle #1  \rrangle}  
\newcommand{\eskp}{\skp{\: , \, }}

\begin{document}
\title[\titelkurz]{\titellangsep}
\author{\ichselber}
\address{\meineaddresse}
\email{\meineemail}
\subjclass[2010]{\mscwert}
%
\begin{abstract}
Let $N$ be a finite group of odd order and $A$ a finite group
  that acts on $N$
  such that
  $\abs{N}$ and $\abs{A}$
  are coprime.
  Isaacs constructed a natural correspondence between
  the set $\Irr_A(N)$ of irreducible complex characters invariant under
  the action of $A$, and the set $\Irr (\C_N(A))$.
  We show that this correspondence preserves Schur indices over
  the rational numbers $\rats$.
  Moreover, suppose that the semidirect product $AN$ is a normal subgroup of
  the finite group $G$ and set $U=\N_G(A)$.
  Let $\chi\in \Irr_A(N)$ and $\chi^{*}\in \Irr (\C_N(A))$
  correspond.
  Then there is a canonical bijection between
  $\Irr(G\mid \chi)$ and $\Irr(U\mid \chi^{*})$ preserving Schur
  indices.
  We also give simplified and more conceptual proofs of (known)
  character
  correspondences above fully ramified sections.
\end{abstract}
\maketitle
%

\section{Introduction}
  Let $G$ be a finite group and
  let $L\subseteq K$ be normal subgroups of $G$.
  Suppose $\phi\in \Irr L$ is fully ramified in $K$.
  This means that $\phi$ is invariant in $K$ and that there is a    unique irreducible character
  $\theta\in \Irr K$ lying above $\theta$.
  This situation occurs naturally in the character theory of
  finite solvable groups, and
  a number of authors
  has studied this situation~\cite{dade76, i73, i82, i83b, lewis96a, lewis97}.
  Under additional conditions,
  there is a subgroup $H\leq G$ with $G=KH$ and $K\cap H=L$
  (see Figure~\ref{fig:bconf}), and
  a correspondence between
  $\Irr(G\mid \phi)$ and
  $\Irr(H\mid \phi)$.
  \begin{figure}[ht]
    \setlength{\unitlength}{0.45ex}
    \centering
    \begin{picture}(70, 55)(-8,-3)
     \put(22.5,2.5){\line(1,1){25}}
     \put(17.5,2.5){\line(-1,1){15}}
     \put(2.5,22.5){\line(1,1){25}}
     \put(47.5,32.5){\line(-1,1){15}}
     \put(18.0,-1){$L$}
     \put(13,-3){$\phi$}
     \put(-2.4,17.2){$K$}
     \put(-7,15.4){$\theta$}
     \put(28.1,48.4){$G$}
     \put(48.0,28.7){$H$}
    \end{picture}
    \caption{}
    \label{fig:bconf}
  \end{figure}
  In particular, Isaacs~\cite{i73} constructs
  such a bijection, when $K/L$ is abelian of odd order.
  He shows that there is a canonical character $\psi$ defined on
  $H/L$, all of whose values are nonzero, and that
  the equation $\chi_H=\psi\xi$ does define a bijection between
  $\chi \in \Irr(G\mid \phi)$ and
  $\xi \in \Irr(H\mid \phi)$.
  The construction of the character $\psi$ is rather lengthy and
  intricate.

  In this paper, we show that the results of Isaacs can be deduced
  from our theory of
  ``magic representations''~\cite{ladisch09diss, ladisch10pre}.
  In fact, this theory arose from
  an attempt to better understand the correspondence of Isaacs.
  The idea is as follows:
  Suppose  $\phi$ is invariant in
  $G$.
  Let $e_{\phi}$ be the central primitive idempotent
  of $\compl L$ associated with $\phi$.
  Since $\phi$ is fully ramified in $K$, we have
  $e_{\phi}=e_{\theta}$, where $\{\theta\}=\Irr(K\mid \phi)$.
  (In fact, this is equivalent to
   $\phi$ being fully ramified in $K$.)
  Set $S= (\compl Ke_{\phi})^L =\C_{\compl Ke_{\phi}}(L)$.
  Then $S\iso\mat_n(\compl)$.
  The factor group $G/L$ acts on $S$.
  Since all automorphisms of $S\iso\mat_n(\compl)$ are inner
  automorphisms,
  there is $\sigma(x)\in S^*$ for each $x\in G/L$ such that
  $s^x = s^{\sigma(x)}$ for all $s\in S$.
  This yields a projective
  representation $\sigma\colon G/L\to S$.
  If we can choose the $\sigma(x)$ such that the restriction of
  $\sigma$ to $H/L$ is an ordinary group representation,
  then we call $\sigma\colon H/L\to S$
  a \emph{magic representation}.
  It is fairly easy to show that
  $\compl G e_{\phi}\iso\mat_n(\compl H e_{\phi})$ when
  a magic representation exists.
  This explains the existence of a character correspondence.
  If $\psi$ is the character of $\sigma$, then
  $\chi_H = \psi \xi$ for corresponding
  $\chi\in \Irr(G\mid \phi)$ and
  $\xi\in \Irr(H\mid \phi)$.

  These results apply to character fives in general.
  They are
  explained in Section~\ref{sec:corr}.
  (Sections~\ref{sec:good}--\ref{sec:fives}
   contain preliminary material.)
  Section~\ref{sec:magic} contains
  results about magic representations
  for character fives.
  In Section~\ref{sec:coprim},
  we  give a very short and easy proof of
  a result including some results of Lewis~\cite{lewis96a, lewis97}.
  In Section~\ref{sec:isaacs}, we show that there is a magic
  representation when $K/L$ is abelian of odd order.
   In Section~\ref{sec:canonical}, we show that there is a
  canonical choice for the magic representation, thereby proving
  the existence of a canonical bijection.
  These two sections yield a new proof of
  Isaacs' result~\cite{i73}.

  The approach described so far works in fact for smaller fields
  than $\compl$, but the field has to contain the values of
  $\phi$.
  In a second part of the paper, we drop the assumption that
   the field contains the
  values of $\phi$.
  We also drop the assumption that $\phi$ is invariant in $G$.
  There is a unique central primitive idempotent, $f$, in
  $\rats L$, such that $\phi(f)\neq 0$.
  Using Clifford theory, one sees that it is no loss of generality
  to assume that $f$ is invariant in $G$.
  This means that the Galois orbit of $\phi$ is invariant in $G$,
  but $\phi$ itself may not be invariant.
  We are able to construct an explicit isomorphism
  $\rats G f\iso\mat_n(\rats Hf)$, when $K/L$ is abelian of odd
  order, and an additional condition
  is  given (Theorem~\ref{t:oddabelianschur}).
  The proof of this result, which occupies
  Sections~\ref{sec:compl} and~\ref{sec:proof},
   may be considered as the heart of this
  paper.
  The proof relies on the approach using magic representations.

  The assumption that $\phi$ is fully ramified in $K$ may be
  skipped.
  The more general result
   follows from
  Theorem~\ref{t:oddabelianschur} by reduction arguments that are more or
  less standard.
  (However, the ``going down'' theorem for semi-invariant
   characters, Proposition~\ref{p:gdsemi},
   might be new.)

  Isaacs~\cite{i73} gave two applications of his study of fully
  ramified sections.
  The first is now known as the Isaacs part
  of the Glauberman-Isaacs correspondence:
  Suppose a group, $A$, acts on another group,
  $N$\kern-0.1em,\/ such that $\abs{A}$ and $\abs{N}$ are relatively prime.
  In case $\abs{N}$ is odd, Isaacs constructed a natural
  correspondence between
  $\Irr_A(N)$, the set of irreducible characters of $N$ invariant
  under the action of $A$, and $\Irr C_N(A)$.
  As an application of our results, we get that this
  correspondence preserves Schur indices over all fields.
  (This is wrong for the Glauberman correspondence, as
   the example of the quaternion group with a $C_3$
   acting on it  shows.)
  Even more is true: Suppose that the semidirect product,
  $AN$, is an invariant subgroup of some finite group $G$.
  Set $U=\N_G(A)$ and $C=\C_N(A)=N\cap U$.
  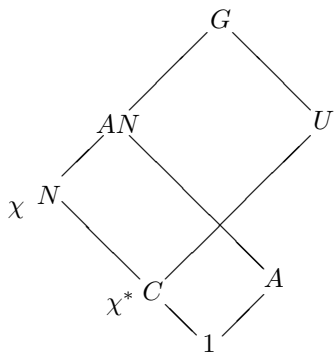
\begin{figure}[ht]
       \setlength{\unitlength}{0.45ex}
       \centering
       \begin{picture}(65, 65)(-8,-11)
          \put(-2.4,17.2){$N$}
          \put(9,31){$AN$}
          \put(31.1,51.4){$G$}
          \put(18.0,-1.6){$C$}
          \put(51.0,31.7){$U$}
          \put(29.6,-11.8){$1$}
          \put(41.5,1){$A$}
          \put(-8,15.4){$\chi$}
          \put(11,-3.5){$\chi^{*}$}
          \put(2.5,22.5){\line(1,1){8}}  
          \put(15.5,35.5){\line(1,1){15}}  
          \put(22.5,2.5){\line(1,1){28}} 
          \put(33.5,-8.5){\line(1,1){8}}  
          \put(17.5,2.5){\line(-1,1){15}} 
          \put(22.5,-2.5){\line(1,-1){6}} 
          \put(15.5,30.5){\line(1,-1){26}} 
          \put(50.5,35.5){\line(-1,1){15}} 
        \end{picture}
        \caption{Above the Isaacs correspondence}
        \label{fig:aboveisaacs}
  \end{figure}
  Let $\chi\in \Irr_A(N)$ and let
  $\chi^{*}\in \Irr C$ be its Isaacs correspondent.
  There is a unique primitive idempotent $i$
  in $(\rats N)^G=\C_{\rats N}(G)$ such that
  $\chi(i)\neq 0$, and a similar defined idempotent $i^{*}$
  in $(\rats C)^U$.
  Then
  $\rats G i \iso \mat_n(\rats U i^{*})$,
  with $n= \chi(1)/\chi^{*}(1)$,
  and there is a canonical correspondence between
  $\Irr(G\mid \chi)$ and $\Irr(U\mid \chi^{*})$.

  In his second application, Isaacs constructed, for a group $G$ of odd
  order,
  a bijection
  between the the set of irreducible characters of $G$  with degree not divisible by
  a given prime $p$, and the set of such characters of the
  normalizer of a Sylow $p$\nbd subgroup, thereby proving the
  McKay conjecture for groups of odd order.
  Let us mention that Turull~\cite{turull06, turull08}
  showed that this bijection preserves Schur indices over any
  field, if $\abs{G}$ is odd.
  His arguments are, however, quite different from those given
  here.

\section{Good elements}\label{sec:good}
We review the concept of \emph{good elements} introduced by
Gallagher~\cite[p.~177]{gall70}.
It is related to
a bilinear form introduced by
Isaacs~\cite{i73}
(Isaacs attributes the form to Dade).
Let $L\nteq G$.
Suppose that
$\phi \in \Irr L$ is invariant in $G$ and let $\crp{F}$ be a field
containing  the values of $\phi$.
Let $e_{\phi}$ be the central primitive idempotent in $\crp{F}L$
associated with $\phi$.
Then any $g\in G$ acts on $\crp{F}Le_{\phi}$ by conjugation.
Since $\crp{F}Le_{\phi}$ is central simple, by the Skolem-Noether
theorem there is $c_g\in (\crp{F}Le_{\phi})^*$ such that
$a^g= a^{c_g}$ for all $a\in \crp{F}Le_{\phi}$.
The element $c_g$ is determined up to multiplication with elements of
$\crp{F}$ by this
property.

If $x,y \in G$ with $[x,y]\in L$ then
$[x,y]e_{\phi}$ and $[c_x,c_y]$ induce the same action
(by conjugation) on $\crp{F} Le_{\phi}$ and so these elements differ by some
scalar. We denote this scalar by  $\skp{ x,y }_{\phi}\in \crp{F}$.
So by definition,
\[ \skp{ x,y }_{\phi}e_{\phi} = [x,y][c_y,c_x] .\]
This definition is independent of the choice of $c_x$ and $c_y$,
since this choice is unique up to multiplication with scalars.

Alternatively, assume that $\phi$ is afforded by a representation
$\rho \colon L \to \mat_{\phi(1)}(\crp{F})$.
For $g\in G$ there is
$\gamma_g\in \mat_{\phi(1)}(\crp{F})$ with
$\rho(l^g)= \rho(l)^{\gamma_g}$ for all $l\in L$.
If $[x,y]\in L$, then
$\rho([x,y])[\gamma_y,\gamma_x]$ centralizes $\rho(L)$,
and thus
it is a scalar matrix. Define $\skp{x,y}_{\phi}$ by
$\skp{x,y}_{\phi}I = \rho([x,y])[\gamma_y,\gamma_x]$.
Since the restriction of $\rho$ to $\crp{F} L e_{\phi}$ is an
isomorphism between $\crp{F} L e_{\phi}$ and
$\mat_{\phi(1)}(\crp{F})$, both definitions agree. From the first
definition we see, however, that
$\skp{x,y}_{\phi}\in \rats(\phi)$, while for the second we have to
assume that $\phi$ is afforded by a representation over $\crp{F}$.
On the other hand the second
definition works for absolutely irreducible
representations over fields of
any characteristic.

Isaacs' definition~\cite[p.~596]{i73} is different, but from the definition
given here it is easier to prove that $\eskp_{\phi}$ is indeed a
bilinear alternating form.
(I learned this definition from
Kn\"{o}rr.)

In most of this work, $\phi$ will be fixed, and so we drop the index if
no confusion can arise.
\begin{lemma}\label{l:skp}
  Let $g$, $x$, $x_1$, $x_2$, $y \in G$ with $[x,y]$, $[x_i,y]\in L$ and
  $l_1$, $l_2 \in L$, and define $\skp{ x,y}_{\phi}=\skp{x,y}$ as above.
   Then
  \begin{enums}
  \item $\skp{ x_1x_2,y }
           = \skp{ x_1,y }
                     \skp{ x_2,y }$.
  \item $\skp{ y,x}
               = \skp{ x,y}^{-1}$.
  \item $\skp{ x l_1, yl_2}
             = \skp{ x,y}$.
  \item $\skp{ x^g, y^g}
               = \skp{ x,y}$.
  \end{enums}
\end{lemma}
\begin{proof}
 All assertions can be
  verified  with routine calculations using commutator identities.
\end{proof}
In particular, $\eskp $ is constant on
cosets of $L$, so we may view $\eskp $ as
being defined on certain elements of $G/L\times G/L$, and we will
do so whenever convenient.
\par
Another trivial remark is this: Suppose that another group, $A$,
acts on $G$ (we use exponential notation $g^a$ for the action)
and stabilizes $\phi$ (that is, $L^a= L$ and $\phi(l^a)=\phi(l)$
for $a\in A$ and $l\in L$).
Then if $x\in G$ and $a\in A$ with $[x,a]=x^{-1}x^a\in L$,
the form $\skp{x,a}_{\phi}$ is still defined. This is clear since we may
work in the semidirect product $AG$, with the usual
identifications of $G$, $L$ and $A$ with subgroups of the
semidirect product. So we will sometimes use
the notation $\skp{x,a}_{\phi}$ in this more general situation
without further explanation.
\par
The  definition of the form given by \index{Isaacs, I.~Martin}Isaacs~\cite[p.~596]{i73} was
from the next lemma for $H= \erz{ L, h}$.
It shows that the form can be computed using only characters.
\begin{lemma}\label{l:cfskp}
  Let $L \leq H \leq G$ and $\chi $ be  a classfunction of $H$ with
  all its irreducible constituents lying over $\phi$. Let
  $h\in H$ and $g\in G$ with $[h,g]\in L$. Then
  $\chi(h^g)= \chi(h) \skp{ h,g}$.
\end{lemma}
\begin{proof}
  We work in the subgroup $\erz{L, h}$ of $H$.
  Writing $\chi_{\erz{L, h}}$ as a linear combination of irreducible
  characters lying above $\phi$, we see that it is no loss to
  assume that $H=\erz{L,h}$, and that
  $\chi$ is irreducible.
  Since $H/L$ is then cyclic and
  $\chi\in \Irr(H\mid\phi)$, in fact
  $\chi$ extends $\phi$.
  Let $\widehat{\rho}\colon  H \to \mat_{\phi(1)}(\compl)$ be a representation
  affording $\chi$ that extends the representation
  $\rho$ affording $\phi$. Choose
  $\gamma_g\in \mat_{\phi(1)}(\compl) $ with
  $\rho(l^g)=\rho(l)^{\gamma_g}$, and let
  $\gamma_h= \widehat{\rho}(h)$.
  Then
  \begin{align*}
    \widehat{\rho}(h^g)
       = \widehat{\rho}(h[h,g])
      &= \widehat{\rho}(h)\rho[h,g] [\gamma_g, \gamma_h][\gamma_h,\gamma_g]
      \\
      &= \widehat{\rho}(h) \skp{h,g}[ \widehat{\rho}(h), \gamma_g]
      \\
      &= \widehat{\rho}(h)^{\gamma_g}\skp{h,g}.
  \end{align*}
  Taking the trace yields the desired result.
\end{proof}
\par
We continue to assume that $\phi\in \Irr L$ is invariant in $G$, where
$L\nteq G$, and $\crp{F}$ is a field containing the values of $\phi$.
We review some known material that yields other ways to
compute the form $\eskp_{\phi}$.
Remember that $A= \crp{F}Ge_{\phi}$ is naturally graded
by the group $G/L$: For $x=Lg\in G/L$, set
$A_x = \crp{F}Le_{\phi}g$
(of course, this is independent of the choice of $g\in x$).
Then $A_xA_y = A_{xy}$ and $A=\bigoplus_{x\in G/L}A_x$.
Now let $S=\C_{A}(A_1)= \C_{A}(L)= (\crp{F}Ge_{\phi})^L$.
The grading of $A$ yields a grading of $S$, namely
$S=\bigoplus_{x\in G/L} S_x$ with $S_x=\C_{A_x}(A_1)$.

A graded unit of $S$ is a unit of $S$ that is contained in some
$S_x$. The set of all graded units of a graded algebra forms a
group.
It is well known that in the situation at hand,
$S_x$ contains units for all $x\in G/L$.
 Namely,  for $g\in G$ there is $c_g\in \crp{F}Le_{\phi}$
with $a^g=a^{c_g}$ for all $a\in \crp{F}Le_{\phi}$, and then
$s_{g}:= c_g^{-1}g = gc_g^{-1} \in \C_{\crp{F}Le_{\phi}g}(L)=S_{Lg}$.
Let $g$, $h\in G$ with $[g,h]\in L$.
Then
\begin{align*}
  [s_{g}, s_{h}]
     = [s_g,h]
     &= [c_g^{-1},h]^g[g,h]
     = [c_g^{-1},c_h]^{c_g}[g,h]
     \\
     &= [c_h,c_g][g,h]
     = \skp{g,h}_{\phi},
\end{align*}
where the first equality follows since $s_g$ and $c_h\in A_1$
commute, and the third equality follows since
$c_g^{-1}$ and $[c_g^{-1}, h]$ are elements of $\crp{F}Le_{\phi}$.
Since $S_{Lg}=\crp{F}s_g$ and similarly for $S_{Lh}$, this is
true for every choice of units
$s_x\in S_{Lg}$, $s_y\in S_{Lh}$.

Suppose we choose a unit $s_x\in S_x$ for every $x\in G/L$.
Then a cocycle $\alpha\colon  G/L \to \crp{F}^*$ is defined
by $s_x s_y = \alpha(x,y)s_{xy}$. We have
\[ \skp{x,y}_{\phi} = [s_x,s_y]= \frac{\alpha(x,y)}{\alpha(y,x)}.\]
(Note that $\alpha$ depends on the choice of the $s_x$, but its
cohomology class does not.)
We have thus proved the following lemma:
\begin{lemma}
\label{l:twisted0}
Hold the notation just introduced, and let $x,y\in G/L$ with
$[x,y]=1_{G/L}$. Then
\[\skp{x,y}_{\phi}e_{\phi} = [s_x,s_y]= s_x^{-1}s_x^y
  \quad\text{and}\quad
  \skp{x,y}_{\phi} = \frac{\alpha(x,y)}{\alpha(y,x)}.
  \]
\end{lemma}
Note that the equation $\skp{x,y}_{\phi}e_{\phi}= s_x^{-1}s_x^y$
is still true if $y$ is an element of some group acting on $G$ and
stabilizing $\phi$, and such that $x\in \C_{G/L}(y)$.
\begin{defi}
 Let $L\leq H\leq G$ and $g\in G$
 (or $g\in \Aut G$, stabilizing $\phi$).
 Then $g$ is called $H$\nbd $\phi$\nbd good if
 $\skp{ c, g }_{\phi}=1 $ for all
 $c \in \C_{H/L}(g)$.
 We drop  $\phi$ if it is clear from
 context. We also drop $H$ if $H=G$.
\end{defi}
By Lemma~\ref{l:skp}, $g\in G$ is ($H$\nbd $\phi$-) good if and only if any other
element of $Lg$ is.
Also if $g$ is $H$\nbd good, then any
$H$\nbd conjugate of $g$ is $H$\nbd good. We can thus speak of good
conjugacy-classes of $G/L$.
Lemma~\ref{l:cfskp} has the following consequence:
\begin{lemma}\label{l:zerogood}
  If $h\in G$ and $\chi\in \compl[\Irr(G\mid \phi)]$ are such that
  $\chi(h)\neq 0$, then $h$ is good for $\phi$.
\end{lemma}
The following result is due to Gallagher~\cite{gall70}:
\begin{lemma}\label{l:gallnrirr}
Let $\phi\in \Irr L$ be invariant in $G$, where $L\nteq G$.
Then
$\abs{\Irr(G\mid \phi)}$ equals the number of $\phi$\nbd good
conjugacy classes of $G/L$.
\end{lemma}

For later use, we prove the following simple
lemma which is essentially due to Isaacs~\cite[p.~600]{i73}:
\begin{lemma}\label{l:cpgood}
  Let $L\nteq G$ and $\phi\in \Irr L$ be invariant in $G$.
  Let $g\in G$ and $K \geq L$. If $g^{m}$ is $K$\nbd good, where
  $(m,\abs{K/L})=1$, then $g$ is $K$\nbd good.
\end{lemma}
\begin{proof}
  For $c\in \C_{K/L}(g)$, we have
  $\skp{c,g}^{\abs{K/L}}= \skp{c^{\abs{K/L}},g}=1
    = \skp{c, g^m}= \skp{c,g}^m$. Thus $\skp{c,g}=1$.
\end{proof}

\section{Character fives}\label{sec:fives}
First we remind the reader of some easy and well known equivalent
conditions for a character to be fully ramified.
\begin{lemma}\label{l:fullyramifiedcond}
 Let $L\nteq K$, $\phi\in \Irr L$ and $\theta\in \Irr(K\mid \phi) $.
 Then the following assertions are equivalent:
 \begin{enumequiv}
 \item $\theta_L= n\phi$ with $n^2 = \abs{K:L}$.
 \item $\phi^K = n\theta $ with $n^2 = \abs{K:L}$.
 \item $\phi$ is invariant in $K$ and
              $\Irr(K\mid \phi)=\{\theta\}$.
 \item $\phi$ is invariant in $K$ and $\theta$ vanishes outside
        $L$.
 \item $e_{\theta}= e_{\phi}$.
 \item \label{ie:fr_onlygood}$\phi$ is invariant in $K$ and $\{L\}$ is the only
         $\phi$\nbd good\index{Good elements} conjugacy class of $K/L$.
 \end{enumequiv}
\end{lemma}
\begin{proof}
  By Lemma~\ref{l:gallnrirr},
  $\abs{\Irr(K\mid \phi)}$ equals
  the number of good conjugacy classes of $K/L$, when
  $\phi$ is invariant in $K$. This yields the
  equivalence of the third and the sixth condition.
  The equivalence of the other conditions is well known and easy
  to establish.
\end{proof}
If $\phi$ has these properties,
we say that $\phi$ is fully ramified in $K$.
We remark that Howlett and
Isaacs~\cite{hi82} have proved,
using the classification of finite simple groups,
that $K/L$ is solvable if some $\phi\in \Irr L$ is fully ramified
in $K$.

An interesting consequence of the last condition of the lemma is
the following:
\begin{cor}\label{c:frval}
  Suppose $\phi\in \Irr L$ is fully ramified in $K$, where
  $K/L$ is abelian.
  Let $e$ be the exponent of $K/L$. Then
  $\rats(\phi)$ contains a primitive $e$-th root of unity.
\end{cor}
\begin{proof}
  Since $K/L$ is abelian, the form
  $\eskp_{\phi} $
  is defined on all of $K/L \times K/L$.
  The last condition of the lemma implies that
  $\eskp_{\phi} $ is a nondegenerate
  alternating form on $K/L \times K/L$. Since it has values in
  $\rats(\phi)^{*}$, this enforces $\rats(\phi)$ to contain a
  primitive $e$-th root of unity.
\end{proof}
\begin{remark*}
  Corollary~\ref{c:frval} is false if $K/L$ is not abelian:
  Namely, let $C$ be a cyclic group of order $p^{a+1}$ and
  let $P$ be the Sylow $p$-subgroup of $\Aut C$.
  Then $\abs{P}=p^a$. Let $K$ be the semidirect product of $P$ and
  $C$. Then it is not difficult to see that
  $L=\Z(K)\subseteq P$ has order $p$ and that the faithful characters of
  $L$ are fully ramified in $K$. Clearly, $K/L$ has exponent
  $p^a$.

  In general, if $p$ is a prime dividing $\abs{K/L}$, then
  $\rats(\phi)$ must contain the $p$-th roots of unity, and that
  is all that can be said.
\end{remark*}
\begin{lemma}\label{l:s-einfach}
  Let $\phi\in \Irr L$ be fully ramified in $K$, where
  $L\nteq K$.
  Let $\crp{F}$ be a field containing $\rats(\phi)$.
  Then $S = (\crp{F} K e_{\phi})^L$
  is central simple with dimension
  $\abs{K/L}$ over $\crp{F}$.
\end{lemma}
\begin{proof}
  $S$ is a twisted group algebra
  of $K/L$ over $\crp{F}$
  (see the discussion before Lemma~\ref{l:twisted0}),
  that is, $S= \bigoplus_{x\in K/L} \crp{F}s_x$ and
  $\dim_{\crp{F}}S = \abs{K/L}$.
  Since $\Irr(K\mid \phi)$ contains only one irreducible
  character, $S$ must be simple.
\end{proof}
The following definition describes the situation we will be
concerned with in this paper:
\begin{defi}\label{d:cfive}
  A \emph{character five} is a  quintuple
  $(G,K,L,\theta, \phi)$ where
  $G$ is a finite group, $L\leq K$ are normal subgroups of $G$,
  and $\phi\in \Irr L$ is fully ramified in $K$,
  and $\{\theta\}=\Irr(K\mid \phi)$.
  Moreover, we assume that $\phi$
  is invariant in $G$.
  An abelian (nilpotent,
  solvable) character five is a character
  five $(G,K,L,\theta,\phi)$ with
  $K/L$ abelian
  (nilpotent, solvable\footnote{Of course,
                       by the before-mentioned result of
                       Howlett and Isaacs~\cite{hi82},
                       every character five is solvable.
                       }).
\end{defi}
The term \emph{character five} is due to Isaacs~\cite{i73}, but
observe that he defines a character five to be abelian, and he
only considers character fives where $K/L$ is abelian.
Since some of our results are valid when $K/L$ is not abelian, we
drop the hypothesis of commutativity of $K/L$ from the
definition of a character five.
We hope
that this change of terminology will not cause too much confusion.

\section{Character correspondences for character fives}
\label{sec:corr}
Now let $(G,K,L,\theta, \phi)$ be a character five and
assume there exists a subgroup $H$
such that $G=HK$ and $L=H\cap K$.
Then $G/K\iso H/L$ canonically.
Let
$\crp{F}$ be a field containing the values of $\phi$ (and thus of $\theta$).
We now review the theory of ``magic
representations''~\cite{ladisch10pre}, that allows to construct
an isomorphism
$\eps\colon \mat_n(\crp{F}He_{\phi}) \to \crp{F}Ge_{\phi}$.

Let $S=(\crp{F}Ke_{\phi})^L$.
By Lemma~\ref{l:s-einfach},
$S$ is central simple.
If $\crp{F}$ is big enough
(for example, if $\theta$ and $\phi$ are afforded by
 $\crp{F}$\nbd representations),
 then $S\iso\mat_n(\crp{F})$.
Assume this and let $E=\{E_{ij}\mid i,j=1,\dotsc,n\}$ be a full set of matrix
units in $S$.
(By this, we mean that $E_{ij}E_{kl}= \delta_{jk}E_{il}$ and
 $1_S= \sum_{i=1}^n E_{ii}$.)
Set $A=\crp{F}Ge_{\phi}$.
By a well known ring theoretic result~\cite[17.4-17.6]{lamMR}, we have that
$A\iso\mat_n(C)$, where $C=\C_A(E)$.
It is clear that $S$ as $\crp{F}$\nbd space is generated by $E$, and thus
$\C_A(E)=\C_A(S)$.

Write $A=\bigoplus_{x\in G/K}A_x$
with $A_{Kg}= \crp{F}Kge_{\phi}$.
This defines a grading of $A$.
The subalgebra $C$ inherits that grading:
for $C_x=\C_{A_x}(S)$, we have
$C=\bigoplus_{x\in G/K}C_x$.
The above isomorphism is one of graded algebras:
\begin{lemma}\label{l:s-matrixr}
  When $S\iso\mat_n(\crp{F})$, then
  $\crp{F}Ge_{\phi}\iso \mat_n(C)$
  as $G/K$-graded algebras, where
  $C=\C_{\crp{F}Ge_{\phi}}(S)$.
\end{lemma}

The group $G$, and even $G/L$ acts on $S= (\crp{F}Ke_{\phi})^L$ by conjugation.
Let $x\in G/L$.
Since $S$ is central simple (by Lemma~\ref{l:s-einfach}),
the Skolem-Noether theorem yields that
there is $\sigma(x)\in S^*$ with
$s^x=s^{\sigma(x)}$ for all $s\in S$. Every such choice of
$\sigma(x)$'s yields a projective representation
$\sigma\colon G/L\to S$. It is unique up to multiplication with a
map $G/L\to \crp{F}^{*}$. We sometimes speak of ``the'' projective
representation associated with the character five
$(G,K,L,\theta,\phi)$.
Let us recall the definition of a ``magic
representation''~\cite{ladisch10pre}:
\begin{defi}\label{d:magic}
  Let $(G,K,L,\theta, \phi)$ be a character five
  and $\crp{F}\supseteq \rats(\phi)$.
  A \emph{magic representation} is a map
  $\sigma \colon H/L \to S = (\crp{F} Ke_{\phi})^L$, where
  $H/L$ is a complement of $K/L$ in $G/L$,
  such that
  \begin{enums}
  \item $\sigma(x)\in S^*$,
  \item $\sigma(xy)=\sigma(x)\sigma(y)$ for all $x$,
         $y\in H/L$ and
  \item $s^x = s^{\sigma(x)}$ for all $s\in S$ and $x\in H/L$.
  \end{enums}
  The character of a magic representation, that is the function
  $\psi \colon H/L \to \crp{F} $ with $\psi(x)= \tr_{S/\crp{F} }(\sigma(x))$, is called
  a magic character.
\end{defi}
Note that a magic representation is determined by the definition
up to multiplication with a linear character of $H/L$.
\begin{thm}\label{t:centiso}
  Let
  $\sigma \colon H/L \to S$ be a magic representation.
   Then the linear map
  \[ \kappa\colon \crp{F} H  \to C=\C_{\crp{F} Ge_{\phi}}(S),\quad
     \text{defined by} \quad
     h \mapsto h \sigma(Lh)^{-1} \text{ for } h\in H ,\]
  is an algebra-homomorphism and induces an isomorphism
  $ \crp{F} H e_{\phi} \iso C$.
  The isomorphism respects the $H/L$-grading of
  $C$ and $\crp{F}He_{\phi}$.
\end{thm}
\begin{proof}
  \cite[Theorem~3.8]{ladisch10pre}
\end{proof}
The reader should note that $\kappa$ restricted to $\crp{F}L$ is
just multiplication with $e_{\phi}$, since
$\sigma(1)=e_{\phi}$.
Using this, the proof of
Theorem~\ref{t:centiso} is straightforward.
\begin{cor}
  If there is a
  magic representation for the character five
  $(G,K,L,\theta, \phi)$ and if
  $(\crp{F}Ke_{\phi})^L\iso \mat_n(\crp{F})$, then
  $\crp{F}Ge_{\phi}\iso \mat_n(\crp{F}He_{\phi})$.
\end{cor}
\begin{thm}\label{t:corr}
  Let $(G,K,L,\theta,\phi)$ be a character five
  such that $S=(\crp{F}Ke_{\phi})^L\iso\mat_n(\crp{F})$,
  where $\rats(\phi)\leq \crp{F}$.
  Every magic representation
  $\sigma \colon H/L \to S^* $ determines
   linear isometries
  $\iota= \iota(\sigma)$ from
  $\compl[\Irr(U\mid \theta)]$
  to $\compl[\Irr(U\cap H\mid \phi)]$
  for all $U$ with $K\leq U\leq G$.
  If $E=\{E_{ij}\mid i, j=1,\dotsc, n\}$ is a
  full set of matrix units in $S$, then
  $\iota(\sigma)$ can be computed by
  \begin{equation}\label{equ:iota}
  \chi^{\iota(\sigma)}(h)
    = \chi(E_{11}\sigma(Lh)^{-1} h).
  \end{equation}
  Write $\iota$ also for the union of these isometries.
  Then $\iota$ commutes with restriction, induction and
  conjugation of class functions, with multiplication by class
  functions of $G/K\iso H/L$,
   and with field automorphisms
  fixing $\crp{F}$,
  and $\iota$ preserves Schur indices of
  irreducible characters over $\crp{F}$.

  Let $\psi$ be the character of $\sigma$,
  and $\chi \in \compl[\Irr( G\mid \theta)]$.
  Then
  \begin{equation}\label{equ:psi}
    \chi_H = \psi \chi^{\iota}
     \quad \text{and}\quad
     (\chi^{\iota})^G = \overline{\psi}\chi.
  \end{equation}
\end{thm}
\begin{proof}
  The theorem is a special case of Theorem~4.3 in \cite{ladisch10pre}.
\end{proof}
\begin{remark}\label{r:corrdet}
  Let $\pi$ be the set of  prime divisors of $\abs{K/L}$.
  If there is any magic representation,
  then there is a magic representation $\sigma$ such that
  $\det \sigma$ has order a $\pi$\nbd number.
\end{remark}
\begin{proof}
  \cite[Remark~4.4]{ladisch10pre}
\end{proof}

\section{Magic representations for character fives}
\label{sec:magic}
\begin{prop}\label{p:psiformel}
Let $L\nteq K$ and $\phi\in \Irr L$ be fully ramified in $K$.
Let $a$ be an element of some group acting on $K$ such
that $\phi^a = \phi$
and choose $\sigma \in S= (\crp{F}K e_{\phi})^L$ with
$s^a = s^{\sigma}$ for all $s\in S$.
Then
\begin{align*}
  \tr(\sigma^{-1})\tr(\sigma)
         &= \sum_{ x\in \C_{K/L}(a)} \skp{ x, a }_{\phi}
         = \begin{cases}
             \abs{\C_{K/L}(a)}  &\text{if $a$ is $K$\nbd good,}\\
             0                  &\text{otherwise.}
           \end{cases}
\end{align*}
\end{prop}
\begin{proof}
  The second equation is clear. Without loss of generality, we can
  assume that $S$ splits, that is $S\iso \mat_n(\crp{F})$.
  The $\crp{F}$\nbd linear map $\kappa$ from $S$ to $S$ sending
  $s$ to $s^a=s^{\sigma}$ has trace
  $\tr(\sigma^{-1})\tr(\sigma)$,
  as an easy computation with matrix units shows.
  Now we use as basis of $S$ a set of graded units
  $s_x\in S\cap \crp{F}Lx$,
  $x\in K/L$.
  If $x\notin \C_{K/L}(a)$ then $(s_x)^a\in \crp{F}s_{x^a}$ is a multiple of another
  basis element and so it
   contributes nothing to the
  trace of $\kappa$. If $x\in \C_{K/L}(a)$,
  then  $s_x^a = \skp{x,a}s_x$ by
  the remark following Lemma~\ref{l:twisted0}.
  The result now follows.
\end{proof}
Applying the proposition to a magic representation
of a character five
yields the absolute value of a magic character.
This generalizes a result of Isaacs~\cite[Theorem~3.5]{i73}.
\begin{cor}\label{c:psiformel}
  If $(G,K,L,\theta,\phi)$ is a character five
  and $\psi$ a magic character of this character five,
  defined on a complement $H/L$ of $K/L$ in $G/L$, then for
  $h\in H$
  \[ \abs{\psi(h)}^2 =
     \begin{cases}
       \abs{\C_{K/L}(h)}   &\text{if $h$ is $K$\nbd good,}\\
       0                   &\text{otherwise.}
     \end{cases}\]
\end{cor}

Next we will show that there is, if the field is big enough,
a finite group $P$, such that
$S=(\crp{F}K e_{\phi})^L \iso \crp{F}Pe_{\mu}$, where
$\mu\in \Lin (\Z(P))$. This follows of course at once from
the theory of projective representations,
but we need to take into account
the action of $G$ on $S$ and so we review this in detail.

Remember that $S$ has a natural grading
$S=\bigoplus_{x\in K/L}S_x$
by the group $K/L$, and that
each component has the form $S_x =\crp{F}s_x$, where
$s_x$ is a unit of $S$.
In particular, $S_1=\crp{F}e_{\phi}=\Z(S)$.
Let
$\Omega= \bigdcup_{x\in K/L} (S_x\cap S^*) = \bigdcup_{x\in K/L} \crp{F}^{*}s_x$
be the set of graded units of $S$.
Then $\Omega$ is a central extension of $\crp{F}^*$ by $K/L$:
\[ \begin{CD}
        1 @>>> \crp{F}^{*} @>>> \Omega @>{\eps}>> K/L @>>> 1
     \end{CD}.
  \]
Following Dade~\cite{dade70a, dade70b},
we call this central extension
the Clifford extension associated with $(K,L,\phi)$ over
 $\crp{F}$.
The epimorphism $\eps$ sends elements of $S_x$ to $x\in K/L$.
Note that for $u$, $v\in \Omega$ with
$[u,v]\in \Ker \eps =\crp{F}^*e_{\phi}$ we have
$ [u,v] = \skp{u^{\eps}, v^{\eps}}_{\phi}$: this is
Lemma~\ref{l:twisted0}.
In particular, if $\phi$ is fully ramified in $K$, then
$\Ker\eps = \Z(\Omega)$.
\par
The group $G$ (even $G/L$) acts on $\Omega$ and centralizes
$\Ker\eps\iso\crp{F}^{*}$.
Obviously, $\Omega$ generates $S$
(as $\crp{F}$\nbd algebra, even as ring), and so we might
realize $S$ as factor algebra of the group algebra
$\crp{F}\Omega$. Of course, $\Omega$ is infinite.
\begin{lemma}\label{l:admtripclosed}
  Hold the above notation and let
  $\crp{F}$ be algebraically closed. Set
  \[ P = \erz{ s\in \Omega \mid
                      \erz{s} \cap \Z(\Omega)= \{e_{\phi}\} \:}. \]
  Then $P$ has the following properties:
  \begin{enums}
  \item \label{i:adm_sur} $P\Z(S)^* = \Omega$
       (equivalently, the restriction of $\eps$ to $P$ is
       surjective),
  \item \label{i:adm_inv}$P^g=P$ for all $g\in G$,
  \item \label{i:adm_fin}$\abs{P}$ is finite and divides $\abs{K/L}^2$.
  \end{enums}
\end{lemma}
\begin{proof}
  Let $x\in K/L$ and $s\in S_x$.
  Then $s^{\ord(x)}\in S_1 =\crp{F}e_{\phi}$,
  so that $s^{\ord(x)}=\lambda e_{\phi}$. Since $\crp{F}$ is
  algebraically closed, there is an $\ord(x)$\nobreakdash-th root,
  $\alpha$, of $\lambda$ in $\crp{F}$.
  Thus  $\alpha^{-1}s$ has
  indeed order $\ord(x)$.
  This holds for any $x\in K/L$ and thus $P$ covers $K/L$.
  It is also clear that $P$ is invariant under $G$.

  To see that $P$ is finite, choose $s_x\in S_x$
  with $\ord(s_x)=\ord(x)$ for any
  $x\in K/L$. Then $P$ is generated by the $s_x$ and
  the $\exp(K/L)$-th roots of unity in $\crp{F}$.
  We have $s_xs_y = \alpha(x,y)s_{xy}$ for some
  $\alpha(x,y)\in \crp{F}$.
  Let $\delta(x)$ be the determinant of
  $s_x$ (as element of $S\iso \mat_n(\crp{F})$).
  Then $\delta(x)^{n}=1$ since $\ord(s_x)=\ord(x)$ divides
  $n=\sqrt{\abs{K/L}}$.
  Since $\delta(x)\delta(y)=\alpha(x,y)^n \delta(xy)$, it follows
  that $\alpha(x,y)^{n^2}=1$.
  But $P\cap\crp{F}^*$ is generated by the values of $\alpha$ and
  thus is finite of order dividing $\abs{K/L}$.
  This finishes the proof.
\end{proof}
\begin{defi}
  Let $(G,K,L,\theta,\phi)$ be a character five
  and $\Omega$ the set of graded units of
  $S=(\crp{F}Ke_{\phi})^L$, where $\crp{F}$ is some field
  containing the values of $\phi$.
  Then an \emph{admissible subgroup} for the character five
  $(G,K,L,\theta, \phi)$ is a subgroup $P\leq \Omega$
  having Properties~\ref{i:adm_sur}--\ref{i:adm_fin}
  of Lemma~\ref{l:admtripclosed}.
\end{defi}
If $\crp{F}$ is not algebraically closed, than an admissible
subgroup may or may not exist. There are, however, other
conditions that ensure the existence of such a group
(see Lemma~\ref{l:pzmu} below).
\begin{lemma}\label{l:adm_mu}
  Let $P$ be an admissible subgroup
  of the character five $(G,K,L,\theta,\phi)$
  and set $Z= P\cap \Z(S)$.
  Let $\mu\colon Z\to \crp{F}^*$ be the restriction of
  $\omega_{\phi}$ (the central character associated with $\phi$)
  to $Z$.
  Then $P/Z\iso K/L$ canonically, $Z=\Z(P)$ and
  the inclusion $P\subset S$ induces an isomorphism
        $\crp{F}P e_{\mu}\iso S$ of $G$\nbd algebras.
\end{lemma}
\begin{proof}
$Z$ is the kernel of $P\to K/L$ and thus $P/Z\iso K/L$.
That $Z=\Z(P)$ follows from $\Z(\Omega)= \crp{F}^*$ and
$\Omega= P \Z(\Omega)$.

Note that $z= ze_{\phi}=\mu(z)e_{\phi}$ for $z\in Z$.
The natural map
$\crp{F}P\to S$ sends the
central idempotent
$e_{\mu}=(1/\abs{Z})\sum_{z\in Z}\mu(z^{-1})z$ of $\crp{F}P$
to $e_{\phi}$, and sends all the other central idempotents of
$\crp{F}Z$ to zero.
As $\crp{F}P\to S$ is clearly surjective,
 $S$ is isomorphic to a factor ring of
  $\crp{F}Pe_{\mu}$, but since
  $\dim_{\crp{F}}S = \abs{K/L}=\abs{P/Z}=
  \dim_{\crp{F}}\crp{F}Pe_{\mu}$, it follows that
  $\crp{F}P e_{\mu}\iso S$.
\end{proof}
Let us illustrate how this can be used.
\begin{prop}\label{p:frcrat}
  Let $(G,K,L,\theta, \phi)$ be a character five
  and suppose that
  $\sigma\colon H/L \to S= (\compl K e_{\phi})^L$ is a magic
  representation.
  Suppose that the order of $x\in H/L$ is relatively prime to $\abs{K/L}$ and
  that $\det(\sigma(x))=1$.
  Then
  $\psi(x)=\tr(\sigma(x))$ is rational.
\end{prop}
\begin{proof}
  Let $P$ be the group defined in
  Lemma~\ref{l:admtripclosed} (over $\compl$) and
  $\mu \in \Lin \Z(P)$ be the linear character defined in
  Lemma~\ref{l:adm_mu}.

  Let $\crp{F}=\rats(\mu)$
  and let $T$ be the $\crp{F}$-subalgebra of
  $S=(\compl K e_{\phi})^L$ generated by $P$, so that
  $T\iso \crp{F}Pe_{\mu}$ naturally
  and $S= \compl T\iso \compl\tensor_{\crp{F}}T$.
  Then $T^x=T$.
  Since $\ord(x)$ is prime to $\dim T$ and $x$ acts on $T$,
  there is a unique
  element $\tau\in T$ such that the following
  conditions hold:
  \[ t^{\tau}= t^x \quad \text{for all $t\in T$}, \quad
     \tau^{\ord(x)}= 1_T \quad \text{and} \quad
     \det(\tau)=1.\]
  The first condition is then in fact true for all $t\in S$, and
  $\tau$ is unique in $S$ subject to these conditions.
  However, $\sigma(x)$ fulfills these conditions, so
  it follows that $\sigma(x)\in T$
   and thus
  $\psi(x)\in \crp{F}$.

  On the other hand, the eigenvalues of $\sigma(x)$ lie in a field $\crp{E}$
  obtained by
  adjoining a primitive $\ord(x)$\nbd th root of unity to $\rats$,
  and thus $\psi(x)\in \crp{E}$.
  Since $\crp{F}$ is obtained from $\rats$ by
  adjoining a $\abs{Z(P)}$\nbd th root of unity, where
  $\abs{Z(P)}$ divides $\abs{K/L}$ (see Lemma~\ref{l:admtripclosed}),
  and since $(\ord(x),\abs{K/L})=1$, we get
  $\crp{F} \cap \crp{E} = \rats$~\cite[Corollary on p.~204]{langAlg}.
  Thus $\psi(x)\in \rats$ as claimed.
\end{proof}

\section[Coprime character fives]{Digression: coprime character fives}
\label{sec:coprim}
\begin{prop}\label{p:coprim}
 Let $(G,K,L,\theta,\phi)$ be a character five
 such that $\abs{G:K}$ and $ \abs{K:L}$ are coprime.
 Then there is $H\leq G$ with $G=HK$ and $H\cap K =L$ and a
 unique magic
 character $\psi$ of $H/L$ of determinant $1$.
 This character  vanishes nowhere,
 so that the equation
 $\chi_H = \psi \xi$ defines an isometry between
 $\compl[\Irr(G\mid \theta)]$ and
 $\compl[\Irr(H\mid\phi)]$.
 Moreover, $\psi$ is rational.
\end{prop}
\begin{proof}
  By the Schur-Zassenhaus Theorem, there is a complement $H/L$ of
  $K/L$ in $G/L$.
  Since $(\abs{H/L},n)=1$, the action of $H/L$ on
  $S=(\rats(\phi)Ke_{\phi})^L$ lifts
  uniquely to a
  magic representation with determinant $1$.
  Let $\psi$ be its character.
   By Lemma~\ref{l:cpgood}, every $h\in H$ is $K$\nbd good
  and thus
  $\psi(h)\neq 0 $ for all $h\in H$
  by Corollary~\ref{c:psiformel}.
  The character correspondence of Theorem~\ref{t:corr} is determined
  by the equation $\chi_H = \psi \xi$ since $\psi$ has no zeros.
  Finally, Proposition~\ref{p:frcrat} yields that $\psi$ is
  rational.
\end{proof}
\begin{remark}
  Suppose $x\in H/L$ has order $p^r$ where $p$ is a prime.
  Let $\omega\in \compl $ be a primitive $p^r$-th root of $1$.
  Then  $\omega-1\in \mathfrak{P}$
   for any prime ideal $\mathfrak{P}$ of $\ints[\omega]$ with
  $\mathfrak{P}\cap \ints = p\ints$.
  It follows that $\psi(x)\equiv \psi(1) \mod \mathfrak{P}$.
  This holds for any character and is well known. Since here
  $\psi(x)$ is rational, we even have that $\psi(x)\equiv
  \psi(1)\mod p$.
  If $p$ is an odd prime, then $\psi(x)$ is completely determined
  by the two conditions
  \[ \psi(x)^2 = \abs{\C_{K/L}(x)}
     \quad\text{and}\quad
      \psi(x)\equiv n \mod p.\]
\end{remark}
We emphasize that we need only the character $\psi$ to compute the
correspondence: The correspondent of $\chi\in \Irr(G\mid \theta)$
is $(1/\psi)  \chi_H$ and the correspondent of
$\xi\in \Irr(H\mid \phi)$ is $(1/\overline{\psi})  \xi^G$.
If $\abs{K/L}$ is odd, even more can be said.
\begin{cor}\label{c:oddcoprim}
  In the situation of Proposition~\ref{p:coprim} assume that
  $\abs{K/L}$ is odd.
  Let $H/L$ be a complement of $K/L$ in $G/L$
  and $\psi$ the unique magic character  with
  $\det \psi = 1$.
  Then for every
  $U/L\leq H/L $ with $\abs{U/L}$ odd,
  $1_U$ is the unique constituent of $\psi_U$ with odd
  multiplicity.
\end{cor}
\begin{proof}
 By Proposition~\ref{p:frcrat} we know that the
 magic character $\psi$ with $\det \psi = 1$
 is rational.
 By Corollary~\ref{c:psiformel},
 $\abs{\psi(h)}^2= \abs{\C_{K/L}(h)}$ for all $h\in H$.
 Since $\abs{K/L}$ is odd,
 $\psi(h)\in \ints$ is odd for all $h\in H$.
 For $U\leq H$ with $\abs{U/L}$ odd,
 let $\beta = \psi_U -1_U$, a generalized character of $U$
 with $L\leq \Ker \beta$.
 For $\tau \in \Irr(U/L)$ we have
 \[ \abs{U/L}(\beta, \tau)_{U/L}
     = \sum_{u \in U/L} \beta(u) \overline{\tau(u)}
     \in 2 \ints  \]
 since $\beta(u)$ is even for all $u\in U/L$.
 As $\abs{U/L}$ is odd, we conclude that
 $(\beta, \tau)_{U/L}$ is even.
 Thus every
 $\tau \in \Irr(U/L)$
 occurs with even multiplicity in $\beta$.
 Thus $1_U$ occurs with odd multiplicity in
 $\psi=1_U+\beta$, while all other constituents occur with even
 multiplicity, as claimed.
\end{proof}
We remark that in the course of the proof we have shown
that $\beta$ can be divided by $2$.
For this we could have appealed to a more general result of
Kn\"{o}rr~\cite[Proposition~1.1(iii)]{kn84}, but for the convenience
of the reader we have repeated the simple argument here.

The following result includes two related results of
Lewis~\cite[Theorem~A in both]{lewis96a, lewis97}:
\begin{cor}
  In the situation of Proposition~\ref{p:coprim}, assume that
  $\abs{G:L}$ is odd, and let $H/L$ be a complement of $K/L$ in $G/L$.
  In the bijection of Proposition~\ref{p:coprim},
  $\chi\in \Irr(G\mid \theta)$ and $\xi\in \Irr(H\mid \phi)$
  correspond if and only if $(\chi_H,\xi)$ is odd.
\end{cor}
\begin{proof}
  It follows from the last result that
  $\psi = 1 +2\gamma$ for some character $\gamma$ of $H/L$.
  From $\chi_H=\psi\xi$ we get
  $\chi_H = \xi +2\gamma\xi$. Thus $\xi$ is the only constituent
  of $\chi_H$ occurring with odd multiplicity.
\end{proof}
In the next section we will see that we can remove the hypothesis
of coprimeness when we add the hypothesis that
$K/L$ is abelian (and odd).

\section{Odd abelian character fives}\label{sec:isaacs}
The main goal of this and the next section is to give  alternative proofs
of some results
due to Isaacs~\cite{i73}.
\begin{thm}\label{t:oddabelian}
  Let $(G,K,L,\theta,\phi)$ be an
  odd abelian character five.\footnote{This means that
  $K/L$ is abelian of odd order}
   Then there is $H\leq G$ with $G= HK$ and $L= H\cap K$,
  such that  every element of $H$ is $K$\nbd good,
   and  there is a
  magic representation $\sigma \colon H/L \to (\compl K e_{\phi})^L$.
\end{thm}
\begin{remark*}
  We will see later that there is even a magic representation
  $\sigma\colon H/L \to (\rats(\phi)Ke_{\phi})^L$.
\end{remark*}
\begin{proof}[Proof of Theorem~\ref{t:oddabelian}]
  We fix some notation needed in the proof.
  Set $S=(\compl K e_{\phi})^{L}$ and let
  $\Omega$ be the group of graded units of $S$
  (with respect to the $K/L$\nbd grading of $S$).
  Let
  \[ P = \erz{ s\in\Omega \mid \erz{s}\cap \Z(\Omega)=1
             }
  \]
  be the subgroup of $\Omega$ defined in Lemma~\ref{l:admtripclosed} and
  set $Z=\Z(P)$.
  Let $\mu\colon Z\to \compl$ be the linear character
  with $z=\mu(z)e_{\phi}$.
   Note that by Lemma~\ref{l:adm_mu}, $S\iso \compl P e_{\mu}$.
  Let $A= \C_{\Aut P}(Z)$. We split the proof into a series of
  lemmas.
  \begin{lemma}\label{l:centrinv1}
    There is $\tau\in A=\C_{\Aut P}(Z)$ such that $\tau$ inverts $P/Z$
    and $\tau^2=1$,
    and such that for $I= \Inn P$ and
    $U=\C_A(\tau)$ we have $A= IU$ and $I\cap U=1$.
  \end{lemma}
  \begin{proof}
    Note that every coset of $Z$ in $P$ contains by definition of
    $P$ an element, $p$, with
    $\ord(p)=\ord(Zp)$.
    Now a result of Isaacs \cite[Corollary~4.3]{i73}
    applies. (The proof is neither long nor difficult.)
  \end{proof}
  Note that the action of $G$ on $P$
  centralizes $\Z(S)$ and thus $Z$, and so we have a homomorphism
  $\kappa\colon G\to A$. Clearly, $L$ is in the kernel of
  $\kappa$.
  The following observation is true for admissible subgroups
  of arbitrary character fives:
  \begin{lemma}\label{l:adm_inn}
    We have $K \kappa = \Inn P$
    and $K\cap \Ker \kappa = L$.
  \end{lemma}
  \begin{proof}
    Let $k\in K$.
    Then there is $p\in P\cap \compl Lk$.
    As $\compl L$ centralizes $S$,
    we have $s^k = s^p$ for all $s\in S$,
    in particular for  $s\in P$.
    Thus $k\kappa\in \Inn P$.
    Conversely, every $p\in P$
    is contained in $\compl Lk$ for some $k\in K$,
    so that the inner automorphism of $P$ induced by $p$
    comes from conjugation with
    $k\in K$.
    This shows $K\kappa = \Inn P$.
    Therefore, $K/(K\cap \Ker \kappa) \iso \Inn P\iso P/Z$.
    It follows that $K\cap \ker\kappa = L$.
  \end{proof}
  We keep the notation $I=\Inn P$ and $U=\C_A(\tau)$, with $\tau$
  as in Lemma~\ref{l:centrinv1}.
  As before, $\kappa\colon G\to A$ is the homomorphism
  induced by the action of $G$ on $P$.
  \begin{lemma}
   Set
   $H= \kappa^{-1}(U)$.
   Then $G=HK$ and $H\cap K=L$.
  \end{lemma}
  \begin{proof}
       Since $( H\cap K ) \kappa \subseteq U\cap I = 1$,
        it follows that $H \cap K \subseteq \Ker \kappa\cap K = L$.
        Thus $H\cap K =L$.
        Since $I=K\kappa\leq G\kappa \leq A= UI$, it follows that
        $G\kappa = (U\cap G\kappa)I$ and thus $G=HK$.
  \end{proof}
  \begin{lemma}\label{l:hallgood}
    All elements of $H$ are $K$\nbd $\phi$\nbd good.
  \end{lemma}
  \begin{proof}
    Let $h\in H$. We have to show
    that $\skp{x,h}_{\phi}=1$ for all
    $x\in \C_{K/L}(h)$.
    First we translate this to a statement about the action of $h$ on
    $P$.
    Let $s_x\in S_x$. Then by Lemma~\ref{l:twisted0},
    $\skp{x,h}_{\phi}= s_x^{-1}s_x^h$.
    We may choose $s_x \in P$, since $P\to K/L$ is surjective.
    Also observe that the isomorphism $P/Z\iso K/L$ sends
    $\C_{P/Z}(h)$ onto $\C_{K/L}(h)$.
    Thus we need to show that $[c,h]=1$ for all $c\in C$,
    where $C\leq P$ is defined by $C/Z=\C_{P/Z}(h)$.
    We may replace $h$ by $u=h\kappa\in U$.
    Since $u^{\tau}=u$, it follows
    $C^{\tau}=C$. Since $\C_{P/Z}(\tau)=1$, we have
    $C=[C,\tau]$ as sets.
    Thus every element in $C$ has the form
    $[d,\tau]=d^{-1}d^{\tau}$ for some $d\in C$.
    For these,
    \begin{align*}
        [d^{-1}d^{\tau},u]
        = [d^{-1},u][d^{\tau},u]
         &= [d,u]^{-1}[d^{\tau},u^{\tau}]
         \\
        &= [d,u]^{-1}[d,u]
         =1,
    \end{align*}
    as was to be shown.
  \end{proof}
  \begin{lemma}\label{l:oddweil}
    There is a representation
    $W\colon U\to \compl Pe_{\mu}$ such that
    $s^{u}=s^{W(u)}$ for all $s\in \compl Pe_{\mu}$.
  \end{lemma}
    This is, in principle, well known. Namely, the group $P$ can be interpreted
    as a Heisenberg group, and $W$ is the corresponding
    Weil representation~\citep[cf.][]{climcnsze00, prasad09p}.
    We give a proof for the sake of completeness and to show that
    the result is neither very deep nor difficult.
  \begin{proof}[Proof of Lemma~\ref{l:oddweil}]
    There is a projective representation
    $\widetilde{W}\colon U \to T=\compl Pe_{\mu}$
    such that $s^u= s^{\widetilde{W}(u)}$ for all $s\in T$.
    Set $t=\widetilde{W}(\tau)$,
    where $\tau$ is as in Lemma~\ref{l:centrinv1}.
    As $\ord(\tau)=2$, we may assume,
    after replacing $t$ by a suitable scalar multiple, that
    $t^2= 1_T$.
    Let $u\in U$.
    Since $\tau^u =\tau$, it follows that
    $t^{\widetilde{W}(u)}$ is a scalar multiple
    of $t$.
    On the other hand, we have
    $(\tr t)^2= \abs{\C_{P/Z}(\tau)}=1$
    by Proposition~\ref{p:psiformel}.
    In particular,
    $\tr(t)\neq 0$.
    It follows $t^{\widetilde{W}(u)}=t$.

    Now let $V$ be a simple $T$\nbd module and set
    \[ V_{+}= \{ v\in V \mid vt=v \}
       \quad \text{and}\quad
       V_{-}= \{v\in V \mid vt=-v\}.\]
    Then, since $t^2=1$, we have
    $V=V_{+}\oplus V_{-}$.
    From the previous paragraph it follows that
    $V_{+}$ and $V_{-}$ are invariant under $\widetilde{W}(u)$.
    Set $d_{+}=\dim V_{+}$
    and $\delta_{+}(u) = \det( \widetilde{W}(u)\text{ on }
    V_{+})$, and define $d_{-}$ and $\delta_{-}(u)$ similarly.
    From
    \[
       \widetilde{W}(u)\widetilde{W}(v)
        = \alpha(u,v)\widetilde{W}(uv)
   \]
    we get
   \[ \delta_{\pm}(u) \delta_{\pm}(v)
       = \alpha(u,v)^{d_{\pm}}\delta_{\pm}(uv).
   \]
    Since $(\tr t)^2 =1$, it follows
    $\tr t = \pm 1 = d_{+}-d_{-}$.
    We may assume $d_{+}-d_{-} =-1$.
    Set
    \[ W(u) =\frac{\delta_{+}(u)}{\delta_{-}(u)}\widetilde{W}(u).\]
    It is now easy to verify that $W$ is multiplicative.
    The lemma follows.
  \end{proof}
    The proof of
   Theorem~\ref{t:oddabelian} is finished by noting
   that
   \[ \begin{CD}
       H/L @>\kappa >> U @>W>> \compl P e_{\mu} @>>> S
   \end{CD}\]
   is the desired magic representation.
\end{proof}

\section{The canonical magic character}
\label{sec:canonical}
In the odd abelian case, it is possible to choose a canonical
$\psi$, as Isaacs has shown.
The existence and the most important properties
of this canonical magic character  can be derived
 from what we have done so far, with (I hope) simpler proofs.
 Some of the arguments we need are taken from the original proof,
 but for the convenience of the
reader and the sake of completeness we repeat them here.
The following is an adaption of Isaacs' definition
of ``canonical''~\cite[Definition~5.2]{i73} to our
purposes.
\begin{defi}\label{d:canonical}
  Let $(G,K,L,\theta,\phi)$ be a character five with $\abs{K/L}$ odd.
  Let $\pi$ be the set of prime divisors of $\abs{K/L}$.
  A
  magic character $\psi\in \ints[\Irr(H/L)]$ is called \emph{canonical} if
  \begin{enums}
    \item $\ord(\psi)$ is a $\pi$\nbd number and
    \item If $p \in \pi$ and $Q\in \Syl_p(H)$, then $1_Q$ is the
    unique irreducible constituent of $\psi_Q$ which appears with
    odd multiplicity.
  \end{enums}
\end{defi}
\begin{remark}
  Let $(G, K, L, \theta, \phi)$ be a
  character five with $\abs{K/L}$ odd.
  If a canonical magic character $\psi\colon H\to \compl$ exists,
  then all $h\in H$ are good.
\end{remark}
\begin{proof}
  Let $h\in H$.
  We have to show that
  $\skp{h,c}_{\phi}=1$ for all $c\in \C_{K/L}(h)$.
  Write $h=\prod_{p} h_p$ as product of its
  $p$-parts.
  Since $\C_{K/L}(h) = \bigcap_{p} \C_{K/L}(h_p)$, we may assume
  that $h$ itself has  prime power order, $p^t$, say.
  If $p$ does not divide $\abs{K/L}$, then $h$ is good by
  Lemma~\ref{l:cpgood}.
  If $p$ divides $\abs{K/L}$, then let $Q\in \Syl_p (H)$ be a
  Sylow $p$-subgroup containing $h$.
  Then, by canonicalness, $\psi_Q = 1_Q + 2\beta$ for some
  character $\beta$.
  It follows that $\psi(h)\neq 0$ and thus $h$ is good by
  Corollary~\ref{c:psiformel}.
\end{proof}
If $K/L$ is not abelian, it may happen that there is no canonical
$\psi$ even if there is a magic character.
For example it may be that there
are $p$\nbd elements in a complement that are not good. An example
where this occurs has been given by Lewis~\cite{lewis96b}. In his
example, $K/L$ is a $p$\nbd group,
 and the complement $H$ is unique up to conjugacy.
\begin{lemma}\label{l:uniquecan}
  The complement $H$ given,
  there is at most one canonical magic character $\psi$.
\end{lemma}
\begin{proof}\cite[p.~610]{i73}
  Suppose $\psi$ and $\psi_1$ are canonical.
  Then $\psi_1 = \lambda\psi$ for some $\lambda\in \Lin (H/L)$.
  If some prime $q\notin \pi$ divides $\ord(\lambda)$,
  then from $\det \psi_1 = \lambda^n \det \psi$ and
  $(q,n)=1$ we conclude that $q $ divides $\ord (\psi_1)$,
  but this contradicts $\psi_1$ being canonical.
  Therefore $\ord(\lambda)$ is a $\pi$\nbd number.
  Let $p\in \pi$ and $Q\in \Syl_p H$.
  Then $[\lambda_Q, (\psi_1)_Q]= [1_Q, \psi_Q]$ and the last is
  odd by the definition of canonical.
  From the assumption that $\psi_1$ is  canonical too we conclude that
  $\lambda_Q=1_Q$. This holds in fact for all $\pi$\nbd subgroups of
  $H$. As $\ord(\lambda)$ is a $\pi$\nbd number, we have
  $\lambda= 1_H$ and $\psi_1=\psi$.
\end{proof}
\begin{thm}\label{t:oddabcan}
  If $(G,K,L,\theta,\phi)$ is an odd abelian character five, then
  there is a  canonical magic character $\psi$.
  Let $H/L$ be the complement of $K/L$ in $G/L$ on which $\psi$ is
  defined.
  For every subgroup $V/L \leq H/L$ with $\abs{V/L}$ odd,
  $1_V$ is the unique irreducible constituent of $\psi$
  which appears with  odd multiplicity.
\end{thm}
\begin{proof}(cf.\ Isaacs~\cite[Theorem~5.3]{i73}.)
  Let $P$, $A$, $\tau \in A$ and $U=\C_A(\tau)$ be as in the proof of
  Theorem~\ref{t:oddabelian}.
  It suffices to show that we may choose
  the
  $W\colon U\to \compl P e_{\mu}$ in Lemma~\ref{l:oddweil}
  such that its character, which we still call $\psi$, is
  canonical.

  First, let $W$ be any representation
  as in Lemma~\ref{l:oddweil}, and let
  $\psi$ be its character.
  We can assume that
  $\psi$ has $\pi$\nbd order.
  As $\tau \in \Z(U)$, we can write
  $\psi = \psi_{+} + \psi_{-}$ where
  $\psi_{+}(\tau)= \psi_{+}(1)$ and
  $\psi_{-}(\tau)= -\psi_{-}(1)$.

  Let $V  \leq U$ with $\abs{V}$ odd and take $v\in V$.
  As $\tau$ centralizes $v$ and $v$ has odd order,
  we have $\tau = (\tau v )^{\ord(v)}$.
  Thus $\C_{P/Z}(\tau v)\subseteq \C_{P/Z}(\tau)=1$.
  Therefore
  \[ 1 = \abs{\psi(\tau v)} = \abs{\psi_{+}(v)- \psi_{-}(v)}\]
  for every $v\in V$.
  This yields
  $(\psi_{+}-\psi_{-}, \psi_{+}-\psi_{-})_V = 1$ and hence
  $(\psi_{+}- \psi_{-})_V = \pm \lambda$
  where $\lambda \in \Lin V$.
  The sign depends  not on $V$,
  but only on wether $\psi_{+}(1)> \psi_{-}(1)$ or
  $\psi_{+}(1)< \psi_{-}(1)$. We conclude
  \[ \psi_V = 2 \gamma_V + \lambda,
     \text{ where }
     \gamma =
     \begin{cases}
       \psi_{-} & \text{if } \psi_{+}(1)> \psi_{-}(1)\\
       \psi_{+} & \text{if } \psi_{+}(1)< \psi_{-}(1).
     \end{cases}\]
  This equation shows that $\lambda$ is the only constituent of
  $\psi_V$ occuring with odd multiplicity.
  Taking determinants in the equation yields
  $\det \psi_V = (\det \gamma_V)^2 \lambda$.
  Thus $\lambda$ can be extended to a linear character of $U$,
  namely to $\mu = \det \psi (\det \gamma)^{-2}$.
  Write $\mu = \mu_{\pi}\mu_{\pi'}$ where $\mu_{\pi}$ is
  the $\pi$\nbd part of $\mu$.
  Then
  $\overline{\mu_{\pi}}\psi$ still has determinantal order a
  $\pi$\nbd number.
  For $Q\in \Syl_p(U)$ where $p\in \pi$ we have
  $(\mu_{\pi})_Q =\mu_Q$ and thus the unique
  irreducible constituent of
  $\overline{\mu_{\pi}}\psi$
  with odd multiplicity is
  $1_Q$. This shows that $\overline{\mu_{\pi}}\psi$ is canonical
  and completes the proof of the existence of a canonical
  magic character.

  Now assume that $\psi$ is canonical.
  We have already seen that for
  $\abs{V}$ odd, $\psi$ has a unique constituent $\lambda$ of odd
  multiplicity and that this constituent is linear.
  To show that $\lambda= 1_V$ it suffices to show
  that $\lambda_Q = 1_Q$
  if $Q$ is a  $p$\nbd subgroup of $V$.
  If $p\in \pi$ this is clear from the definition of canonical.
  If $p \notin \pi$, then $(\abs{Q}, \abs{P/Z})=1$
  and the result follows from
  Corollary~\ref{c:oddcoprim}, applied to $Q$.
\end{proof}
As in the coprime case, we get as a corollary:
\begin{cor}
  Let $(G,K,L,\theta,\phi)$ be an abelian character five with
  $\abs{G:L}$ odd.
  Then there is a complement $H/L$ of $K/L$ in $G/L$
  and a bijection between $\Irr(G\mid \theta)$ and
  $\Irr(H\mid \phi)$ where
  $\chi$ and $\xi$ correspond if and only if
  $(\chi_H, \xi)$ is odd.
\end{cor}
We conclude this section with some results showing that the
canonical magic representation has values in
$(\rats(\phi)Ke_{\phi})^L$.

If $\alpha$ is an automorphism
of the field $\crp{E}$,
then $\alpha$ acts on  the group algebra $\crp{E} G$
by acting on coefficients.
If $\sigma\colon H/L\to (\crp{E}K e_{\phi})^L$ is a magic
representation, then we write $\sigma^{\alpha}$ for the
map
$H/L \ni x \mapsto \sigma(x)^{\alpha} \in
(\crp{E}Ke_{\phi^{\alpha}})^L$.
Now the following is easy to
verify~\cite[Proposition~4.5]{ladisch10pre}:
\begin{prop}\label{p:fieldsmag}
  Let $(G,K,L,\theta,\phi)$ be a character five
  with magic representation $\sigma$ and magic character
  $\psi$, and let $\alpha$ be a field
  automorphism.
  Then $\sigma^{\alpha}$ is a magic representation with magic
  character $\psi^{\alpha}$
  for the character five
  $(G,K,L,\theta^{\alpha}, \phi^{\alpha})$.
\end{prop}
The definition of a canonical character is invariant under field
automorphisms. Thus:
\begin{cor}
  Let $(G,K,L,\theta,\phi)$ be an odd abelian character five with
  canonical magic character $\psi$, and let
  $\alpha$ be a field automorphism.
  Then $\psi^{\alpha}$ is the
  canonical magic character associated with the character five
  $(G,K,L,\theta^{\alpha},\phi^{\alpha})$.
\end{cor}
\begin{cor}\label{c:magicrepvals}
 The image of the canonical magic representation
 is contained in $(\rats(\phi)Ke_{\phi})^L$,
and the values of the canonical character are in
  $\rats(\phi)$.
\end{cor}
\begin{proof}
  Let $\sigma\colon H/L\to (\compl K e_{\phi})^L$ be the canonical
  magic representation.
  Let $\alpha$ be a field automorphism fixing $\rats(\phi)$.
  Thus $\phi^{\alpha}=\phi$.
  By Proposition~\ref{p:fieldsmag}, $\sigma^{\alpha}$ is a magic
  representation for the character five $(G,K,L,\theta,\phi)$.
  Since $\sigma^{\alpha}$ is canonical too, we have
  $\sigma^{\alpha}=\sigma$.
  Since this holds for all $\alpha$ centralizing $\rats(\phi)$,
  it follows that
  $\sigma(Lh)\in \rats(\phi) K$, as was to be shown.
  The second assertion follows from the first.
\end{proof}
\begin{remark*}
  In fact, one can show that the values of the canonical character
  are contained in a much smaller
  field~\cites[Cor.~6.4]{i73}[or][Cor.~4.35]{ladisch09diss}.
\end{remark*}
It follows that Theorem~\ref{t:corr} applies with
$\crp{F}=\rats(\phi)$, if we know that
$(\rats(\phi)Ke_{\phi})^L \iso
\mat_n(\rats(\phi))$.

\section{Semi-invariant characters}\label{sec:semi-inv}
We review the notion of semi-invariant characters and recall some
results that we need.
Throughout this section, let $L\nteq G$ and $\phi\in \Irr L$.
Let $\crp{F}$ be a field of
characteristic zero. All characters are assumed to take values in
some field containing $\crp{F}$, so that expressions like
$\crp{F}(\phi)$ are defined.
We need the following well known fact.
\begin{lemma}\label{l:Traceidempotent}
  \[ e = \T_{\crp{F}}^{\crp{F}(\phi)}(e_{\phi})
    := \sum_{\alpha\in \Gal(\crp{F}(\phi)/\crp{F}) }
        (e_{\phi})^{\alpha}\]
  is
  the unique central primitive idempotent of\/ $\crp{F}L$
  for which $\phi(\crp{F}L e)\neq 0$.
\end{lemma}
\begin{notat}
 We write $e_{ (\phi, \crp{F}) }$ for the idempotent of
 Lemma~\ref{l:Traceidempotent}. In particular, if
 $\crp{F}=\crp{F}(\phi)$, then
 $e_{ (\phi, \crp{F}) }= e_{\phi}$.
\end{notat}
\begin{lemma}\label{l:fe_iso}
  \[  \crp{F}G_{\phi}e_{ (\phi, \crp{F}) } \ni a \mapsto
      ae_{\phi} \in \crp{F}(\phi)G_{\phi}e_{\phi} \]
  is an isomorphism of\/ $\crp{F}$\nbd algebras.
\end{lemma}
\begin{proof}
  \cite[Lemma~5.3]{ladisch10pre}
\end{proof}
The following notation will be convenient:
 Let $L \nteq G$ and $e \in \Z(\crp{F}L)$ be a primitive idempotent.
 Let $G_e= \{g\in G \mid e^g =e\}$ and write $e^G$ for the
 idempotent defined by
 $e^G:= \T_{G_e}^G(e) = \sum_{g\in [G:G_e]}e^g$.
 Finally, given an idempotent $e$ of $\Z(\crp{F}L)$, we set
 \[ \Irr(G\mid e) = \{\chi\in \Irr G \mid \chi(e)\neq  0\}.\]
The following result is also well known:
\begin{prop}\label{p:cliffordcorr}
  Set  $e= e_{ (\phi, \crp{F}) }$ and $f= e^G$
  and let $T=G_e$ be the inertia group of $e$.
  Then\/ $\crp{F}Gf \iso \mat_{\abs{G:T}}(\crp{F}Te)$.
  Induction defines a bijection between\/
  $\Irr(T\mid e)$ and\/ $\Irr(G\mid f)$ that preserves field
  of values and Schur indices over\/ $\crp{F}$.
\end{prop}
In general, $G_{\phi}$ may be smaller than $G_e=T$.
For $\xi\in \Irr(G_{\phi}\mid \phi)$, the field
$\crp{F}(\xi^{T})$ is  contained  in $\crp{F}(\xi)$, but may be
strictly smaller. If this happens, the Schur index of $\xi^T$
over $\crp{F}(\xi^T)$ may be bigger
than that of $\xi$.
\begin{defi}
  Let $L\nteq G$ and $\phi\in \Irr L$.
  We say that $\phi$ is \emph{semi-invariant} in $G$
  over the field $\crp{F}$,
  if $e_{(\phi, \crp{F})}$ is invariant in $G$.
  If $\phi$ is semi-invariant over $\rats$, then we  say
  it is semi-invariant.
\end{defi}
This definition is equivalent to the one given
by Isaacs~\cite[Definition 1.1]{i81b}
(cf.~\cite[Lemma~5.6]{ladisch10pre}).
\begin{lemma}\label{l:fnsg}
  Let $L\nteq G$ and $\phi\in \Irr L$ be
  semi-invariant over\/ $\crp{F}$.
  Set $\Gamma=\Gal(\crp{F}(\phi)/\crp{F})$.
  \begin{enums}
  \item For every $g\in G$ there is a unique
        $\alpha_g\in \Gamma$ such that
        $\phi^{g \alpha_g}=\phi$.
  \item The map $g\mapsto \alpha_g$ is a group homomorphism from
        $G$ into $\Gamma$ with kernel $G_{\phi}$.
  \item \label{i:centchgalois}
        For $g\in G$ and $z\in \Z(\crp{F}L )$ we have
        $\omega_{\phi}(z^g) = \omega_{\phi}(z)^{\alpha_g}$, where
        $\omega_{\phi}\colon \Z(\crp{F}L)\to \crp{F}(\phi)$ is the
        central character associated with $\phi$.
  \end{enums}
\end{lemma}
\begin{proof}
  \cite[Lemma~2.1]{i81b},\cite[Lemma~5.7]{ladisch10pre}
\end{proof}

\section{Main theorem}\label{sec:main}
For convenience, we introduce some terminology.
\begin{defi}
  A quintuple $(G,K,L,\theta, \phi)$ is called a semi-invariant
  character five, if
  $G$ is a finite group, $L\leq K$ are normal subgroups of $G$,
  and the characters
  $\phi\in \Irr L$ and $\theta \in \Irr K$ are fully ramified
  with respect to each other
  and semi-invariant in $G$.
\end{defi}
As the attentive reader will have remarked,
this terminology is  inconsistent with
Definition~\ref{d:cfive}, since a
semi-invariant character five is not necessarily a
character five in the sense of Definition~\ref{d:cfive}.
It would have been more consistent to speak of
``invariant character fives'',
``semi-invariant character fives'' and
``character fives'' (not necessarily semi-invariant).
To avoid any ambiguity, we will speak of invariant/semi-invariant
character fives from now on.

We will need one further hypothesis.
\begin{defi}\label{d:strong}
  A semi-invariant character five
  $(G,K,L,\theta, \phi)$
  with $K/L$ abelian is said to be
  \emph{strongly controlled}
  if there is $N\nteq G$, such that the following hold:
  \begin{enums}
  \item $K\leq N\leq G_{\phi}$,
  \item $\C_{K/L}(N)= 1$,
  \item $(\abs{N/\C_N(K/L)}, \abs{K/L})=1$.
  \end{enums}
\end{defi}
Thus the subgroup of the automorphism group of $K/L$ induced by
$N$ acts coprimely and fixed point freely on $K/L$.
I do not know whether such an assumption is really necessary for
the next result. On the other hand, the assumption that
$\abs{K/L}$ is odd \emph{is} necessary, even
for strongly controlled character fives.
\begin{thm}
  \label{t:oddabelianschur}
  Let $(G,K,L,\theta,\phi)$ be a strongly controlled
  semi-\hspace{0pt}invariant
  character five, such that
  $K/L$ is abelian of odd order.
    Set $f= e_{(\phi, \rats)}$.
  Then there is $H\leq G$ such that
  $KH=G$, $K\cap H = L$, every element of $H_{\phi}$ is
  $K$-$\phi$-good,
  and $\rats G f \iso \mat_n(\rats H f)$
  as $G/K$\nbd graded algebras
  (with $n=\sqrt{\abs{K/L}} = \theta(1)/\phi(1)$).
\end{thm}
Suppose $K\leq U\leq G$ and set $V=U\cap H$.
The isomorphism of the theorem restricts to an isomorphism
$\rats Uf\iso\mat_n(\rats V f)$.
These isomorphisms yield character correspondences.
(See the discussion in~\cite[Section~2]{ladisch10pre}
 or~\cite[Theorem~3.4]{marcus08}.)
While there is no canonical choice for the
isomorphism of Theorem~\ref{t:oddabelianschur}, there
are choices that lead to a canonical choice of the bijection
between
$\compl[\Irr( G \mid f )]$
and
$\compl[\Irr(H\mid f)]$.
(The isomorphism of Theorem~\ref{t:oddabelianschur}
 is unique up to (inner) isomorphisms
 of $(\rats K f)^L\iso \mat_n(\rats(\phi))$.)
\begin{prop}
  \label{p:oddabschur_sup}
  Assume the situation of Theorem~\ref{t:oddabelianschur}.
  For every $U\leq G$ with $K\leq U$, there is an isomorphism
  \[\iota\colon \compl[\Irr(U\mid f)]
    \to \compl[U\cap H \mid f].
  \]
  The union $\iota$
  of these isomorphisms has the following properties:
  \begin{enums}
  \item \label{i:c_irr}
        $\chi \in \Irr(U\mid f)$ if and only if
        $\chi^{\iota} \in \Irr(U\cap H\mid f)$.
  \item \label{i:c_isometr}$(\chi_1^{\iota}, \chi_2^{\iota})_{U\cap H}
         = (\chi_1, \chi_2)_U$ for
         $\chi_1$, $\chi_2\in \compl[\Irr(U\mid f)]$.
  \item \label{i:c_prop}$\chi(1)= n \chi^{\iota}(1)$.
  \item \label{i:c_resuv}
        $(\chi_{U_1})^{\iota} = (\chi^{\iota})_{U_1\cap H}$
        for $K\leq U_1\leq U_2\leq G$ and
        $\chi\in \compl[\Irr(U_2\mid f)]$.
  \item \label{i:c_induv}$(\tau^{U_2})^{\iota}= (\tau^{\iota})^{U_2\cap H}$ for
        $K\leq U_1 \leq U_2 \leq G$ and
        $\tau \in \compl[\Irr(U_1\mid f)]$.
  \item \label{i:c_conj}$(\chi^{\iota})^h = (\chi^h)^{\iota}$ for $h\in H$.
  \item \label{i:c_lin}$(\beta\chi)^{\iota} = \beta \chi^{\iota}$
        for all $\beta \in \compl[\Irr(U/K)]$.
  \item \label{i:c_aut}If $\alpha$ is a field automorphism, then
        $(\chi^{\alpha})^{\iota} = (\chi^{\iota})^{\alpha}$;
        and \label{i:c_field}
        $\rats(\chi)= \rats(\chi^{\iota})$.
  \item \label{i:c_brauer}
        $[\chi]_{\rats} = [\chi^{\iota}]_{\rats}$
        for $\chi\in \Irr(U\mid f)$.
  \item \label{i:c_res} If $U\leq G_{\phi}$ and
        $\chi \in \Irr( U\mid \phi)$, then
        $\chi_{U\cap H} =\psi\chi^{\iota}$,
        where $\psi$ is the canonical magic character
        associated with the invariant character five
        $(G_{\phi}, K, L, \theta, \phi)$.
  \end{enums}
  Moreover, $\iota$  is  determined uniquely by some of these properties
  (namely, by linearity and~\ref{i:c_induv},~\ref{i:c_aut} and~\ref{i:c_res}).
\end{prop}
\begin{proof}[Proof (Uniqueness)]
  Let $\iota$ and $\iota_1$ be two such isometries and take
   a subgroup $U$ with $K\leq U\leq G$
  and
  $\chi\in \Irr(U\mid \phi)$.
  By Clifford correspondence,
  $\chi= \tau^U$ for a unique $\tau\in \Irr(U_{\phi}\mid \phi)$.
  By \ref{i:c_induv}, it follows
  that $\chi^{\iota}   =(\tau^U)^{\iota}
                         = (\tau^{\iota})^U$ and
  similarly
  for $\iota_1$.
  By~\ref{i:c_res},
  we have
  $\tau^{\iota} = (1/\psi)\tau_{U_{\phi}\cap H}= \tau^{\iota_1}$.
  Thus $\chi^{\iota}=\chi^{\iota_1}$.
  Finally, for $\alpha\in\Aut(\rats(\phi))$ any element
  of $\Irr(U\mid \phi^{\alpha})$
  has the form $\chi^{\alpha}$ for some
  $\chi\in \Irr(U\mid\phi)$, and we have
  $(\chi^{\alpha})^{\iota}
    =(\chi^{\iota})^{\alpha}
    =(\chi^{\iota_1})^{\alpha}
   = (\chi^{\alpha})^{\iota_1}$.
  Since
  $\Irr(U\mid f)= \bigcup_{\alpha\in \Aut(\rats(\phi))}
                   \Irr(U\mid \phi^{\alpha})$,
  it follows that $\iota=\iota_1$.
  This shows uniqueness.
\end{proof}
The existence will be proved in Section~\ref{sec:proof}, together
with Theorem~\ref{t:oddabelianschur}.
In Section~\ref{sec:appl},
we use these results to show that the
Isaacs half of the Glauberman-Isaacs correspondence preserves
Schur indices (among other things).

\section{Existence and uniqueness of the complement}
\label{sec:compl}
In this section, we show that
if a semi-invariant character five is strongly controlled, then
there is a complement $H/L$ to $K/L$ in $G/L$,
and that  $H$ is determined up to
conjugacy by some weak additional condition.
This result will not be needed in the proof of
Theorem~\ref{t:oddabelianschur} and
Proposition~\ref{p:oddabschur_sup}. In fact, in the proof of the
latter results, we will give another construction of
the supplement $H$.

We need a general lemma
about the bilinear form associated with $\phi$.
\begin{lemma}\label{l:skpconj}
 Let $L\nteq G$, let $\phi\in \Irr G$ and
 $x,y\in G_{\phi}$ with $[x,y]\in L$.
 \begin{enums}
 \item If $g\in G$, then
       $\skp{ x^g, y^g }_{\phi^g }= \skp{x, y}_{\phi}$.
 \item If $\alpha \in \Aut \rats(\phi)$, then
       $\skp{x,y}_{\phi^{\alpha}}= \skp{x,y}_{\phi}^{\alpha}$.
 \end{enums}
\end{lemma}
\begin{proof}
  Let $\crp{E}=\rats(\phi)$.
  Remember that
  $\skp{x,y}_{\phi}e_{\phi} = [x,y][c_y, c_x]$, where
  $c_x\in \crp{E}Le_{\phi}$ is such that
  $a^x = a^{c_x}$ for all $a\in \crp{E}Le_{\phi}$, and similar for
  $c_y$.
  If $g\in G$, then $c_x^g\in (\crp{E}L e_{\phi})^g = \crp{E}L
  e_{\phi^g}$.
  Any  $b\in \crp{E}Le_{\phi^{g}}$ can be written as $b=a^g$
  with $a\in \crp{E}Le_{\phi}$.
  Thus
  \[b^{c_x^g} = (a^g)^{c_x^g}= a^{c_x g } = a^{xg}= a^{gx^g}=
  b^{x^g}.
  \]
  It follows that
  \begin{align*}
    \skp{ x^g, y^g}_{\phi^g} e_{\phi^g}
      &= [x^g, y^g][c_y^g, c_x^g]
       = ([x,y][c_y, c_x])^g
      \\
      &= (\skp{x,y}_{\phi} e_{\phi})^g
       = \skp{x,y}_{\phi} e_{\phi^g}.
  \end{align*}
  The first assertion follows.
  The proof of the second is similar:
  We may extend  $\alpha_{\crp{E}}$ naturally to
  an automorphism of $\crp{E}G$, acting trivially on $G$.
   Then we get
  \begin{align*}
    \skp{x,y}_{\phi^{\alpha}}e_{\phi^{\alpha}}
      &= [x,y][c_y^{\alpha}, c_x^{\alpha}]
       = ([x,y][c_y, c_x])^{\alpha}
      \\
      &= (\skp{x,y}_{\phi} e_{\phi})^{\alpha}
       = \skp{x,y}_{\phi}^{\alpha} e_{\phi^{\alpha}}.
  \end{align*}
  The proof follows.
\end{proof}
The arguments in the proof of the next result extend
those of Isaacs~\cite[p.~304--5]{i82}
for invariant character fives.
\begin{prop}\label{p:existenceuniquenessofh}
  Let $(G,K,L,\theta, \phi)$ be a strongly controlled
  character five with $K/L$ abelian.
  Then there is a unique conjugacy class of subgroups
  $H\leq G$ such that
  $HK=G$, $H\cap K=L$ and
  every element of $H\cap \C_{N}(K/L)$ is $K$-$\phi$-good.
\end{prop}
\begin{proof}
  Let $C=\C_N(K/L)$. Observe that then
  $\eskp_{\phi}$ is defined on
  $C/L \times K/L$. Let
  \[ B=\{ c\in C\mid \skp{c, k}_{\phi}= 1
                 \text{ for all } k\in K\}.\]
  We claim that $B\nteq G$.
  Let $b\in B$ and $g\in G$.
  There is $\alpha\in
  \Aut(\rats(\phi))$ such that
  $\phi^{\alpha g} = \phi $.
  Let $k\in K$. Using both parts of
  Lemma~\ref{l:skpconj}, we get
  \begin{align*}
    \skp{b^g, k^g}_{\phi}
      &= \skp{b^g, k^g }_{\phi^{\alpha g}}
       = \skp{ b, k }_{\phi^{\alpha}}
       = \skp{ b, k }_{\phi}^{\alpha}
       =1.
  \end{align*}
  Since $k^g\in K^g=K$ was arbitrary, it follows that $b^g\in B$.
  This establishes that $B\nteq G$.

  Via $\eskp_{\phi}$,
  the factor group $C/B$ is isomorphic to a subgroup of
  $\Lin(K/L)$, and so $\abs{C/B}\leq \abs{K/L}$.
  Since $\phi$ is fully ramified in $K/L$,
  the form $\eskp_{\phi}$ is
  nondegenerate on $K/L$,
  and thus we have $B\cap K = L$.
  Therefore $\abs{K/L}= \abs{BK/B}\leq \abs{C/B}$.
  It follows $BK=C$ and $C/B\iso K/L$.
  Since $\abs{C/B}=\abs{K/L}$ and $\abs{N/C}$ are coprime, the
  group $C/B$ has a complement, $M/B$, in $N/C$.
  Let $H= \N_G(M)$.
  By the Frattini-argument,
  $G=HN= HMC =H BK =HK $.
  Moreover, we have $[H\cap K, M]\leq K\cap M = L$, and so
  $(H\cap K)/L \leq \C_{K/L}(M)=\C_{K/L}(N/C)=1$, so that
  $H\cap K = L$.
  Since $K/L\iso C/B$ as group with $M$-action, the same argument
  shows that
  $H\cap C= B$.
  By definition of $B$, every element of $B$ is $K$-$\phi$-good.

  For uniqueness, assume that $U$ is
  another subgroup having the properties
  in the proposition.
  Then $U\cap C\leq B$, since $U\cap C$ is good.
  Since $C= C\cap UK = (C\cap U)K$, it follows that $C\cap U =B$.
  Since $N= N\cap UC = (N\cap U)C$, it follows that $(N\cap U)/B$
  is a complement of $C/B$ in $N/B$.
  By the conjugacy part of the Schur-Zassenhaus Theorem,
  $N\cap U= M^c$ with $c\in C$.
  It follows that $H^c = \N_G(M^c)\geq U$, and thus
  $\N_G(M^c)=U$ as claimed.
\end{proof}
In the special case where
$(\abs{N/K}, \abs{K/L})=1$,
the supplement $H$ is determined up to conjugacy by the properties
$HK=G$ and $H\cap K =L$.
This can be proved using standard, purely group theoretical
arguments and is well known.

\section[Magic crossed representations]{Magic crossed representations for semi-invariant character fives}
\label{sec:magiccrossed}
We need to review the theory
of magic crossed representations, which we will use to prove the
results of Section~\ref{sec:main}.
For convenient reference, let us fix the following assumptions and
notation:
\begin{hyp}\label{hyp:sef}\hfill
 \begin{enums}
 \item $(G,K,L,\theta,\phi)$ is a semi-invariant character five,
 \item $f=e_{(\phi,\rats)}$, $\crp{E}=\Z(\rats K f)$  and
       $S=(\rats K f)^L$,
 \item for $g\in G$, let $ \alpha_g\in \Aut\crp{E}$ be the
       automorphism of $\crp{E}$ induced by conjugation with $g$,
       and set $\crp{F}=\crp{E}^G$.
 \end{enums}
\end{hyp}
The isomorphism $\rats K f\iso \rats(\phi)K e_{\phi}$
of Lemma~\ref{l:fe_iso},
given by $a\mapsto ae_{\phi}$,
sends
$\Z(\rats K f) $ onto
$\Z(\rats(\phi)K e_{\phi})=\rats(\phi)e_{\phi}$.
The same is true for
the centers of $\rats L f$ and $S$.
Thus
$\Z(\rats K f)=\Z(\rats L f) = \Z(S)\iso \rats(\phi)$.
The combination of Lemma~\ref{l:fe_iso} and
Lemma~\ref{l:s-einfach} yields that $S$ is central simple over
$\crp{E}$.

$G$ acts on $\crp{E}$ by conjugation.
By Lemma~\ref{l:fnsg}, Part~\ref{i:centchgalois}, the kernel of
this action is $G_{\phi}$, the inertia group of $\phi$.
Thus we have proved:
\begin{lemma}\label{l:centers}
  We have
  $\crp{E}=\Z(\rats L f) = \Z(S)\iso \rats(\phi)$,
  and $S$ is central simple over $\crp{E}$.
  $G$ acts on $\crp{E}$ with kernel $G_{\phi}$.\hfill \qedsymbol
\end{lemma}
To prove Theorem~\ref{t:oddabelianschur},
we will procede as follows:
First, we show that
$S\iso \mat_n(\crp{E})$.
Thus we find a full set of
matrix units, $E$, in $S$.
Then $\rats G f \iso \mat_n(C)$, where
$C=\C_{\rats G f}(E)$ is
the centralizer of $E$ in $\rats G f$.
Second, we show that
$C\iso \rats H f$, where $H/L$ is a complement of $K/L$ in $G/L$.
To do this, we need a magic crossed representation, which
generalizes a magic representation
to the semi-invariant case~\cite{ladisch10pre}.
We review this concept now.

Clearly, $\Aut \crp{E}$ acts naturally on $\mat_n(\crp{E})$ by
acting on the entries of a matrix.
Thus, if an isomorphism $S\iso\mat_n(\crp{E})$ is given,
it yields an action of
$\Aut\crp{E}$ on $S$.
To be more concrete,
let $\{e_{ij}\mid 1\leq i,j\leq n\}$ be a set of matrix units in
$S$.
For $\alpha\in \Aut \crp{E}$, define
$\widehat{\alpha}\in \Aut S$ by
$\left(\sum_{i,j} s_{ij}e_{ij}\right)^{\widehat{\alpha}}
        = \sum_{i,j} s_{ij}^{\alpha}e_{ij}$ for
$s_{ij}\in \crp{E}$.
For convenient reference, we summarize our assumptions and notation:
\begin{hyp}\label{hyp:mat}
 Assume Hypothesis~\ref{hyp:sef} and the following:
 \begin{enums}
 \item $H\leq G$ with $G=HK$ and $H\cap K=L$,
 \item $S\iso \mat_n(\crp{E})$,
 \item $\Aut\crp{E}\ni \alpha \mapsto \widehat{\alpha}\in \Aut S$
       denotes the action of $\Aut\crp{E}$ on $S$ with respect to
       a fixed set of matrix units $E$ in $S$.
 \end{enums}
\end{hyp}
Now let us recall the  definition of a magic crossed representation,
adapted to our situation:
\begin{defi}
  In the situation of~\ref{hyp:mat},
  $\sigma\colon H/L \to S$ is a
  \emph{magic}  crossed
  representation
  (with respect to $h \mapsto \widehat{\alpha_h}$),
  if, for all $u$, $v\in H/L$ and $s\in S$, we have
  \[ \sigma(u)^{\widehat{\alpha_v}}\sigma(v)= \sigma(uv)
         \quad \text{and}\quad
         s^u = s^{\widehat{\alpha_v}\sigma(u)}.
  \]
\end{defi}
If $\sigma\colon H/L \to S$ is a magic crossed
representation,
  then the linear map
  \[ \kappa\colon \rats H  \to C=\C_{\rats G f}(E),\quad
     \text{defined by} \quad
     h \mapsto  h\sigma(Lh)^{-1} \text{ for } h\in H ,\]
  is an algebra-homomorphism and induces an isomorphism
  $ \rats H f
   \iso C$~\cite[Theorem~6.11]{ladisch10pre}.
Thus, since $\rats G f \iso \mat_n(C)$,
Theorem~\ref{t:oddabelianschur} will follow if we can show
that there is a magic crossed representation.

Moreover, every magic crossed representation determines a character
correspondence $\iota(\sigma)$ between
$\Irr(U\mid f)$ and
$\Irr(U\cap H \mid f)$ for each $K\leq U\leq G$ having
Properties~\ref{i:c_irr} to~\ref{i:c_brauer} of
Proposition~\ref{p:oddabschur_sup}~\cite[Theorem~6.13]{ladisch10pre}.
Finally, if $\sigma \colon H/L \to S$ is a magic crossed representation,
then the map
$\sigma_{\phi}\colon H_{\phi} \to (\rats(\phi)K e_{\phi})^L$
defined by $\sigma_{\phi}(Lh) =\sigma(Lh)e_{\phi}$ is a magic
representation
for the invariant character five
$(G_{\phi}, K, L, \theta, \phi)$, and for
$\chi\in \Irr(U_{\phi}\mid \phi)$ we have
$(\chi^U)^{\iota(\sigma)}
  = (\chi^{\iota(\sigma_{\phi})}
    )^{U\cap H}$~\cite[Proposition~6.15]{ladisch10pre}.
Here $\iota(\sigma_{\phi})$ means the character correspondence of
Theorem~\ref{t:corr}.
Thus to prove Theorem~\ref{t:oddabelianschur} and
Proposition~\ref{p:oddabschur_sup},
it suffices to show the following:
\begin{thm}\label{t:sigmacan}
  Assume the situation of
  Theorem~\ref{t:oddabelianschur}.
  Then the assumptions of Hypothesis~\ref{hyp:mat} hold,
  and
  there is a magic crossed representation
  $\sigma\colon H/L \to S$ such that
  $\sigma_{\phi}\colon H_{\phi}/L \to (\rats(\phi)K e_{\phi})^L$,
  defined by
  $\sigma_{\phi}(h)=\sigma(h)e_{\phi}$, is the canonical magic
  representation of the invariant character five
  $(G_{\phi}, K, L, \theta, \phi)$.
\end{thm}

\section{Proof of main theorem}\label{sec:proof}
The proof of Theorem~\ref{t:sigmacan} mimics that of
Theorem~\ref{t:oddabelian}.
Assume Hypothesis~\ref{hyp:sef}.
Then $S=(\rats K f)^L$ admits a natural grading by  the factor group $K/L$:
\[S= \bigoplus_{x\in K/L} S_x
  \quad \text{with}\quad
  S_{Lk}= S\cap \rats Lk.
\]
Moreover, every component $S_x$ contains units of $S$.
All this follows from the corresponding results for invariant
character fives, via Lemma~\ref{l:fe_iso}.
   Let
  \[\Omega = \bigdcup_{x\in K/L} (S_x\cap S^*)\]
  be the set of graded units of $S$.
 Then we have a central extension
  \[ \begin{CD}
       1 @>>> \crp{E}^* @>>> \Omega @>\eps >> K/L @>>> 1.
     \end{CD}
  \]
  The group $G$ acts on $\Omega$.
  Since $\phi$ is only semi-invariant, the action of $G$ on
  $\Z(\Omega)=\crp{E}^*$ may be nontrivial.

Now assume the situation of Theorem~\ref{t:sigmacan}.
Remember that we are given
a semi-invariant character five, $(G,K,L,\theta,\phi)$,
with $K/L$ abelian of odd order,
which is strongly controlled, that is,
there is $K\leq N\nteq G$ such that $N/\C_N(K/L)$ acts coprimely and
fixed point freely on $K/L$.
First we exhibit a subgroup $P\leq \Omega$ with similar
properties as in Lemma~\ref{l:admtripclosed}.
This is the only part of the proof of
Theorem~\ref{t:sigmacan} where we use the assumption that the
character five is strongly controlled.
The main idea in the proof of the following lemma
is taken from a paper of Turull~\cite{turull08}.
\begin{lemma}\label{l:pzmu}
  Set  \[ P = [\Omega, N]
    \quad\text{and}\quad
    Z= P\cap \crp{E}.\]
  Then the following hold:
  \begin{enums}
  \item \label{i:pzmu_sur}$P \Z(S)^* = \Omega$
        (equivalently, $P\to K/L$ is surjective).
  \item \label{i:pzmu_inv}$P^g=P$ for all $g\in G$.
  \item \label{i:pzmu_fin}$\abs{P}$ is finite and has the same exponent as $K/L$.
  \item \label{i:pzmu_coset}Every coset of $Z$ in $P$ contains an element
        $u$ with $Z\cap \erz{u}= \{1\}$.
\end{enums}
\end{lemma}
\begin{proof}
  Since $N\nteq G$, it follows that $G$ normalizes $P$,
  this is~\ref{i:pzmu_inv}.

  The group $N$ acts on $\Omega$ and centralizes
  $\crp{E}=\Z(S)$, since $N\leq G_{\phi}$.
  Since $N/\C_N(K/L)$ acts
  coprimely and
  fixed point freely on $K/L\iso\Omega/\crp{E}^*$, we
  have
  $[\Omega/\crp{E}^*, N]=\Omega/\crp{E}^*$.
  It follows that $[\Omega, N]\crp{E}^{*} = \Omega$, as claimed.

  Let $k$ be an odd natural number.
  The set
  $\{x\in K/L\mid x^k=1\}$ is a characteristic subgroup of $K/L$
  and thus normalized by $N$.
  Let \[\Omega_k= \{ u\in \Omega \mid
                      u^k \in \crp{E} \} \]
  be the pre-image in $\Omega$.
  Since $N$ acts coprimely and fixed point freely on
  $\Omega_k/\crp{E}^*$,
  it follows that $[\Omega_k, N]\crp{E}^* =\Omega_k$.
  Thus
  \[P\cap \Omega_k
     = P\cap [\Omega_k, N]\crp{E}^*
     = [\Omega_k, N](P\cap \crp{E}^*)
     = [\Omega_k ,N] Z.
  \]
  Next we claim that $u^k= 1$ for every
  $u\in [\Omega_k,N]$.

  First, let $s,t \in \Omega_k$ be arbitrary.
  Remember that
  $[\phantom{s}, \phantom{i}]\colon \Omega \times \Omega \to
  \crp{E}^*$ is bilinear.
  Thus
  \begin{align*}
    (st)^k &= s^{k} t^k[ s, t]^{\binom{k}{2}}
               = s^k t^k [s^{\binom{k}{2}}, t]
              = s^k t^k
  \end{align*}
  as $k$ is odd.

  Let $a\in N$ and $s\in \Omega_k$.
  We apply the last equation to $s^{-1}$ and
  $t= s^a$:
  \[ [s, a]^k = (s^{-1} s^a)^k = s^{-k}(s^a)^k = s^{-k} (s^k)^a = 1,\]
  as $a\in N$ centralizes $s^k\in \crp{E}^*$.
  It follows that $[\Omega_k,N]$ is generated by elements of order
  $k$.
  Since we saw $(st)^k = s^k t^k$ before,
  it follows that $u^k=1$ for all
  $u\in [\Omega_k,N]$, as claimed.

  If we take for $k$ the exponent of $K/L$, then $\Omega_k=\Omega$
  and  thus $P=[\Omega,N]$ has
  the same exponent as $K/L$.
  It follows that
  $Z$ is finite (and cyclic).
  Thus $\abs{P}= \abs{K/L}\abs{Z}$ is finite, too.

  Now let $v\in P$ be arbitrary and set $k= \ord(Zv)$.
  We have seen earlier that
  $P\cap \Omega_k = [\Omega_k, N]Z$. Thus
  there is  $u\in Zv \cap [\Omega_k, N]$, and we have seen before
  that $u$ has order $k$.
  This means that $Z\cap \erz{u} = 1$,
  which shows~\ref{i:pzmu_coset}.
\end{proof}
\begin{lemma}\label{l:adm_semi}
  Let $P\leq \Omega$ be a group
  satisfying Properties~\ref{i:pzmu_sur}--\ref{i:pzmu_fin} of
  Lemma~\ref{l:pzmu}, and set $Z= P \cap \crp{E}$.
  Let $\mu\colon Z\to \rats(\phi)$ be the restriction
  of the central character of $\phi$ to $Z$.
  Then
  \begin{enums}
  \item \label{i:adms_inv}
        The set $\{ue_{\phi}\mid u \in P\}\subseteq
        (\rats(\phi)Ke_{\phi})^L$ is an admissible subgroup for
        the character five $(G_{\phi}, K, L, \theta, \phi)$.
  \item \label{i:adms_pzisokl}$Z=\Z(P)$ and $P/Z\iso K/L$ as groups with
        $G$\nbd action.
  \item \label{i:adms_siso} $S\iso \rats(\phi)Pe_{\mu}$.
  \item \label{i:adms_uvcom}The commutator map defines a
        non-degenerate alternating form $P/Z\times P/Z \to Z$.
  \end{enums}
\end{lemma}
\begin{proof}
  Part~\ref{i:adms_inv} follows from applying the isomorphism of
  Lemma~\ref{l:fe_iso}.
  Part~\ref{i:adms_siso} and $Z=\Z(P)$ then follows from
  Lemma~\ref{l:adm_mu}.
  That the isomorphism $P/Z\iso K/L$ respects the action of $G$ is
  clear.
  Finally, the form
  $\eskp \colon K/L\times K/L \to
  \rats(\phi)$ is non-degenerate. Let $x$, $y\in K/L$ be the
  images of $u$, $v\in P$ under the canonical homomorphism
  $\eps\colon P\to K/L$. Then, by Lemma~\ref{l:twisted0},
  \[ \skp{x,y}_{\phi}e_{\phi}= [ue_{\phi}, ve_{\phi}]
      = [u,v]e_{\phi} = \mu([u,v])e_{\phi}.\]
  It follows that $(u,v)\mapsto \mu([u,v])$ is non-degenerate, and
  thus $[\phantom{x},\phantom{x}]$ itself is nondegenerate.
\end{proof}
In the next result, $\Aut S$ denotes the set of the ring automorphisms of $S$,
which
are the automorphisms of $S$ as $\rats$\nbd algebra. We could also
work with the automorphisms of $S$ as $\crp{F}$\nbd algebra
(where $\crp{F}=\C_{\crp{E}}(G)$).
\begin{lemma}\label{l:adm_innsemi}
  Let $P$ be as in Lemma~\ref{l:pzmu}.
  The action of $G$ on $S$ defines an homomorphism
  \[ \kappa\colon G \to A := \{ \alpha\in \Aut S \mid
                P^{\alpha} = P  \}\]
  with $L$ in the kernel.
  Let
  \[I= \{ \alpha \in \C_{A}(\crp{E}) \mid
            \alpha_{|P} \in \Inn P\}.\]
  Then $ K\kappa=I \iso \Inn P$.
\end{lemma}
\begin{proof}
  The first assertion is clear.

  The map $\C_{A}(\crp{E}) \to \Aut P$, $\alpha\mapsto \alpha_{|P}$,
  is injective, since
  $\crp{E}$ and $P$ generate $S$ as ring and $A \leq \Aut S$.
  By definition, $I$ maps into $\Inn P$.
  Conversely, every inner automorphism of $P$ induces an inner automorphism
  of $S$, simply since $P\leq S^*$, and thus centralizes
  $\crp{E}=\Z(S)$. Thus $I\iso \Inn P$.
  If $k\in K$, then for every unit  $u\in S_{Lk}$
  we have $s^k = s^u$ for all $s\in S$, since $L$ centralizes $S$.
  It follows $K\kappa =I$.
\end{proof}
\begin{lemma}
  Let $A$ and $I$ be as in Lemma~\ref{l:adm_innsemi}. Then $I$ is
  the kernel of the natural map
  $A\to \Aut(P/Z)\times \Aut \crp{E}$.
\end{lemma}
\begin{proof}
   Since $P/Z\iso K/L$ is abelian, inner automorphisms of
   $P$ centralize $P/Z$. Thus $I$ is in the kernel of
  $A \to \Aut(P/Z)\times \Aut(\crp{E})$.
  Conversely, suppose $\alpha\in A$ acts trivially on $P/Z$ and
  on $\crp{E}$. Then $\alpha$ centralizes also $Z\subset \crp{E}$.
  It is known~\citep[Lemma~4.2]{i73} and not difficult to show
  that an automorphism of $P$ centralizing $P/Z$ and $Z$ is
  inner.
  (Here one needs that $P$ is a *-group in the sense
   of Isaacs~\cite[Def~4.1]{i73},
   which follows from Lemma~\ref{l:pzmu},
   Part~\ref{i:pzmu_coset}.)
  Thus $\alpha\in I$ as claimed.
\end{proof}
The next lemma generalizes Lemma~\ref{l:centrinv1} to our situation,
the proof is nearly the same.
\begin{lemma}\label{l:centralinvolution}
  There is $\tau\in A=\N_{\Aut S}(P)$ such that
  $\tau$ inverts $P/Z$, centralizes $\crp{E}$, $\tau^2=1$, and
  $U=\C_{A}(\tau)$ is a complement of $I$ in $A$.
\end{lemma}
\begin{proof}
  There is $\tau_0\in \Aut P$ of order $2$, inverting $P/Z$ and
  centralizing $Z$
  (Lemma~\ref{l:centrinv1} \cite[cf.][Lemma~4.3]{i73}).
  Since $S\iso \rats(\phi)P e_{\mu}$,
  this $\tau_0$ can be extended to an automorphism $\tau$ of $S$
  of order $2$ and centralizing $\crp{E}$.

  Observe that $\tau$ maps to a central element of
  $\Aut(P/Z)\times \Aut \crp{E}$. Thus $I\tau \in \Z(A/I)$ and so
   $\erz{I, \tau}\nteq A$.
  Since $I\iso P/Z$ has odd order,
  $\erz{\tau}\in \Syl_2\erz{I, \tau}$, and thus, by the Frattini
  argument, $A = I\C_{A}(\tau)$. As $\tau$ inverts $P/Z$,
  it follows $\C_I(\tau) =1$, as desired.
\end{proof}
\begin{cor}
  Let $H=\kappa^{-1}(U)$. Then $G=HK$ and $H\cap K=L$.
  Every element of $H_{\phi}$ is $K$\nbd good.
\end{cor}
\begin{proof}
  From $G\kappa \leq A=UI$ and $I=K\kappa \leq G\kappa$ it follows
  $G\kappa = (G\kappa \cap U)I$ and thus
  $G=HK$. If $k\in H\cap K$, then
  $k\kappa \in U\cap I=1$, and thus $s^k=s$ for all $s\in S$. It
  follows $k\in L$.

  That elements of $H_{\phi}$ are good follows from
  the corresponding result for
  invariant $\phi$ (Lemma~\ref{l:hallgood}).
\end{proof}
Note that by Proposition~\ref{p:existenceuniquenessofh}, the complement
$H$ is unique up to conjugacy. We will prove
Theorem~\ref{t:sigmacan} for the group $H$
of the last corollary.
We now work toward finding a suitable set of matrix units in $S$.
\begin{lemma}
  Let $R=\{ r\in P\mid r^{\tau} = r^{-1}\}$ with $\tau$ the automorphism of
  Lemma~\ref{l:centralinvolution}.
  Then $P = \bigdcup_{r\in R}Zr$.
\end{lemma}
\begin{proof}
  For arbitrary $x\in P$ one has $x^{\tau} = zx^{-1}$ for some
    $z\in Z$. There is a unique $d\in Z$ with $d^2=z$ as $Z$ has
    odd order.
    For this $d$ we get $d^{-1}x\in R$. Thus every coset of $Z$ contains
    exactly one element of $R$, as claimed.
\end{proof}
\begin{lemma}
  There are abelian subgroups $X, Y\leq P$
  with $X\cap Y = Z$ and $XY=P$.
  Set
  \[ e = \frac{1}{n}\sum_{x\in X\cap R} x\in S.\]
  Then $e^{\tau} = e$ and
  \[ \{ E_{r,s}= r^{-1}e s\mid r,s\in R\cap Y\}\]
  is a full set of matrix units in $S$.
  (Thus $S\iso \mat_n(\crp{E})$.)
\end{lemma}
\begin{proof}
  Remember that
  $[\phantom{x},\phantom{x}]\colon P/Z\times P/Z \to Z$ is a
  nondegenerate alternating form (Lemma~\ref{l:adm_semi}\ref{i:adms_uvcom}).
  Choose two maximal isotropic subspaces, $X/Z$ and $Y/Z$, with
  $P/Z = X/Z \times Y/Z$. Then $X$ and $Y$ are abelian and we have
  $X\cap Y=Z$ and $XY=P$.

  As $X$ is abelian, it follows that
  $R\cap X$ is a subgroup:
  We have $(rs)^\tau = r^{-1}s^{-1}= (rs)^{-1}$ for
  $r,s\in R\cap X$.
  The order of $X\cap R$ is $\abs{X/Z}=n$. It follows that $e$ is
  an idempotent.
  That $e^{\tau}=e$ is clear.

  Next, let $y\in Y\setminus Z$. We claim that
  $e^y e=0$.
  Note that the group algebra $\crp{E}[X\cap R]$ is contained
  canonically in $S$ as subalgebra, since
  $S= \bigoplus_{r\in R} \crp{E}r$.
  We may view $e$ and
  \[e^y = \frac{1}{n} \sum_{x\in X\cap R} x^y
        = \frac{1}{n} \sum_{x\in X\cap R} x[x,y]
        \]
  as idempotents in $\crp{E}[X\cap R]$, since $[x,y]\in \crp{E}$.
  If $y\in Y\setminus Z$, then
  $x\mapsto [x,y]\in \crp{E}^*$ is a nontrivial group homomorphism
  from $X\cap R$ to $\crp{E}^*$.
  It follows that $ e^y e=0$ and $eye =0$.
  Thus we get
  \[ E_{r,s}E_{u,v} = r^{-1}s e^s e^u u^{-1}v
                    = \delta_{s,u}r^{-1}e v
                    =\delta_{s,u} E_{r,v} \]
  for $r,s,u,v\in Y\cap R$.
  We also get
  \[ \sum_{r\in Y\cap R} E_{r,r}
      =\sum_{r\in Y\cap R} \frac{1}{n} \sum_{x\in X\cap R} x[x,r]
      = \frac{1}{n}\sum_{x\in X\cap R } x \sum_{r\in Y\cap R} [x,r]
      =1_S.\]
  The result follows.
\end{proof}
\newcommand{\halpha}{\widehat{\alpha}}
Note that the isomorphism $S\iso\mat_n(\crp{E})$
can be used to define an action of
$\Aut \crp{E}$ on $S$.
The point about the next lemma is that
the corresponding action homomorphism has image in
$U\leq \Aut S$, where
$U=\C_A(\tau)$, as defined in Lemma~\ref{l:centralinvolution}.
\begin{lemma}
    For $\alpha\in \Aut \crp{E}$, define
    $\halpha\in \Aut S$ by
    \[ \left( \sum_{r,s\in Y\cap R} c_{r,s} E_{r,s}\right)^{\halpha}
       = \sum_{r,s\in Y\cap R} c_{r,s}^{\alpha} E_{r,s}
       \quad \text{for}\quad
       c_{r,s}\in \crp{E}.\]
    Then $\alpha\mapsto \halpha$ is a monomorphism
    from $\Aut \crp{E}$ into $U=\C_{A}(\tau)$,
    and $\widehat{\alpha}_{\crp{E}}=\alpha$.
\end{lemma}
\begin{proof}
  It is clear that $\alpha\mapsto \halpha$ is a monomorphism
  into $\Aut S$.
  From
  \[ \left( \sum_{ r,s\in Y\cap R } c_{r,s} E_{r,s}
     \right)^{\tau}
     = \sum_{ r,s \in Y\cap R } c_{r,s} E_{r^{-1}, s^{-1}}\]
  we see that $\halpha \tau = \tau \halpha$.
  Thus $\halpha\in \C_{\Aut S}(\tau)$.
  It remains to show that
  $\halpha\in A$, that is, $\halpha$ maps $P$
  onto itself. We do this by showing that $\halpha$ maps
  $Z$, $Y\cap R$ and $X\cap R$ onto itself.

  It is clear that $\halpha$ maps $Z$ onto itself, since
  $\halpha_{|\crp{E}}=\alpha$ and $Z$ is a finite subgroup of
  $\crp{E}^*$.

  Let $y$ and $r\in Y\cap R$.
  Then $y e^r = yr^{-1}er = E_{ry^{-1}, r}$.
  Thus
  \[ y^{\halpha} = \left( \sum_{r\in Y\cap R} ye^r
                   \right)^{\halpha}
                 = \left( \sum_{r\in Y\cap R} E_{ry^{-1}, r}
                   \right)^{\halpha}
                 = \sum_{r\in Y\cap R} E_{ry^{-1}, r}
                 = y, \]
  so in fact $\halpha$ centralizes $Y\cap R$.

  Now let $x\in X\cap R$ and $r\in Y\cap R$.
  Then
  \begin{align*}
    xe^r &= \frac{1}{n} \sum_{u\in X\cap R} xu[u,r]
          = \frac{1}{n} \sum_{u\in X\cap R} u [x^{-1}u, r]\\
         &= [x^{-1},r] \frac{1}{n} \sum_{u\in X\cap R} u[u,r]
          = [r,x]e^r,
  \end{align*}
  with $[r,x]\in Z$. (Remember that the commutator map is
  bilinear in both variables.)
  As $Z$ is a finite subgroup of $\crp{E}$, there is
  $k\in \nats$ with $z^{\alpha}=z^k$ for all $z\in Z$.
  Thus
  \begin{align*}
    x^{\halpha}
    &= \left(  \sum_{ r\in R\cap Y } xe^r
       \right)^{\halpha}
     = \sum_{r\in R\cap Y} ([r,x]e^r)^{\halpha}
     = \sum_{r\in R\cap Y} [r,x]^k e^r\\
    &= \sum_{r\in R\cap Y} [r, x^k]e^r
     = x^k.
  \end{align*}
  Thus $\halpha$ maps $X\cap R$ onto itself.
  This finishes the proof that $P^{\halpha}=P$.
\end{proof}
Let $U_{\phi}=\C_{U}(\crp{E})$
with $U=\C_A(\tau)$ as in Lemma~\ref{l:centralinvolution}, and observe that
$H_{\phi} = \kappa^{-1}(U_{\phi})$.
For $u\in U$, we denote by $\alpha_u$ the restriction of $u$ to
$\crp{E}=\Z(S)$.
To prove Theorem~\ref{t:sigmacan}, we will show the following:
\begin{lemma}\label{l:sigmacan2}
      There is a map $\sigma\colon U\to S$ such that
      for all $u$, $v\in U$
      \[ \sigma(u)^{\widehat{\alpha_v}}\sigma(v)
         = \sigma(uv)
         \qquad \text{and}\qquad
         s^u = s^{\widehat{\alpha_u} \sigma(u)}
         \quad \text{for all $s\in S$},\]
      and such that
      $\sigma_{U_{\phi}}\colon U_{\phi}\to S$ is canonical in the
      sense of Definition~\ref{d:canonical}.
\end{lemma}
Then Theorem~\ref{t:sigmacan} follows by
composing $\sigma$ and $\kappa\colon H\to U$.
\begin{proof}[Proof of Lemma~\ref{l:sigmacan2}]
  From the results of Section~\ref{sec:isaacs} it follows that there
  is an homomorphism
  $\sigma_{\phi}\colon U_{\phi}\to (\compl K e_{\phi})^L$ that is magic in
  the sense that $s^u = s^{\sigma_{\phi}(u)}$ for all
  $s\in (\compl K e_{\phi})^L$ and
  $u\in   U_{\phi}$ (Lemma~\ref{l:oddweil}),
  and we may assume that $\sigma_{\phi}$ is canonical in the
  sense of Definition~\ref{d:canonical} by the results of
  Section~\ref{sec:canonical}.
  By Corollary~\ref{c:magicrepvals}, the image of $\sigma $ is
  contained in $(\rats(\phi)K e_{\phi})^L$.
  Now remember that
  $S \ni s \mapsto
   se_{\phi} \in (\rats(\phi)K e_{\phi})^L$
  is an isomorphism (Lemma~\ref{l:fe_iso}).
  We get a unique homomorphism
  $\sigma_{U_{\phi}}\colon U_{\phi} \to S$ such that
  $\sigma_{U_{\phi}}(u)e_{\phi} = \sigma_{\phi}(u)$
  for all $u\in U_{\phi}$.
  Moreover, for $s\in S$ and $u\in U_{\phi}$ we have
  $s^u = s^{\sigma_{U_{\phi}}(u)}$, again
  by Lemma~\ref{l:fe_iso}.
  We must extend $\sigma_{U_{\phi}}$ to
  a magic crossed representation of $U$.

  For $u\in U$ and $z\in \crp{E}$, we have
  $z^{\widehat{\alpha_u}}=z^u$ and
  thus $\widehat{\alpha_u}^{-1}u \in U_{\phi}$.
  Now define
  \[ \sigma(u) := \sigma_{U_{\phi}}( \widehat{\alpha_u}^{-1} u).\]
  Since $U_{\phi}= \Ker( u\mapsto \alpha_u)$, the map $\sigma$
  extends $\sigma_{U_{\phi}}$.
  For $s\in S$,
  \[ s^{\widehat{\alpha_u}\sigma(u)}
      = s^{\widehat{\alpha_u} \sigma_{U_{\phi}}(\widehat{\alpha_u}^{-1} u)}
      = s^{\widehat{\alpha_u} (\widehat{\alpha_u}^{-1}u)}
      = s^u.\]
  To see that $\sigma$ is a crossed representation, we need the
  following fact:
  \begin{equation}\label{equ:canon_inv}
  \sigma_{U_{\phi}}(u)^{a} = \sigma_{U_{\phi}}(u^{a})
     \quad \text{for arbitrary $u\in U_{\phi}$ and $a\in U$}.
  \tag{*}
  \end{equation}
  Assuming this for the moment, we see that
  \begin{align*}
    \sigma(u)^{\widehat{\alpha_v}}\sigma(v)
      &= \sigma_{U_{\phi}}\left( \widehat{\alpha_u}^{-1} u\right)^{\widehat{\alpha_v}}
         \sigma_{U_{\phi}}\left( \widehat{\alpha_v}^{-1}v\right)
      \\
      &= \sigma_{U_{\phi}} \left( \left(\widehat{\alpha_u}^{-1}u
                         \right)^{\widehat{\alpha_v}}
                  \right)
         \sigma_{U_{\phi}}\left( \widehat{\alpha_v}^{-1} v
                 \right)
       \\
      &= \sigma_{U_{\phi}} \left( \widehat{\alpha_v}^{-1}
                         \widehat{\alpha_u}^{-1}
                         uv
                  \right)
       = \sigma_{U_{\phi}}\left( \widehat{\alpha_{uv}}^{-1} uv
                 \right)
       = \sigma(uv),
  \end{align*}
  where the second equation follows from~\eqref{equ:canon_inv}, applied
  to $a=\widehat{\alpha_v}\in U$.

  To establish~\eqref{equ:canon_inv},
  view  $a\in U$ as fixed and consider the
    map $u\mapsto (\sigma_{U_{\phi}})^a(u)= {\sigma_{U_{\phi}}}(u^{a^{-1}})^{a}$.
    We will show that $(\sigma_{U_{\phi}})^a$ is also a canonical magic representation.
    From uniqueness it will then follow that
    $(\sigma_{U_{\phi}})^a={\sigma_{U_{\phi}}}$,
   that is, $\sigma_{U_{\phi}}(u)^a= (\sigma_{U_{\phi}})^a(u^a)= {\sigma_{U_{\phi}}}(u^a)$ as claimed.

    Clearly $(\sigma_{U_{\phi}})^a$ is a homomorphism.

    Let $s^a\in S$. Then
    \[ (s^a)^{(\sigma_{U_{\phi}})^a(u)} = (s^a)^{{\sigma_{U_{\phi}}}(aua^{-1})^a}
          = (s^{{\sigma_{U_{\phi}}}(aua^{-1})})^a
          = (s^{aua^{-1}})^a = (s^a)^u. \]
    Thus $(\sigma_{U_{\phi}})^a$ is magic.

    Let $\psi$ be the character of ${\sigma_{U_{\phi}}}$ with values
    in $\crp{E}$ on which $U$ acts.
    Then $\psi^a$ defined by
    $\psi^a(u^a) = \psi(u)^a$ is the character of $(\sigma_{U_{\phi}})^a$.
    The definition of ``canonicalness'' (Definition~\ref{d:canonical})
    is invariant to conjugation by group automorphisms and field
    automorphisms, and thus $\psi^a$ is canonical.
    It follows that $(\sigma_{U_{\phi}})^a$ is magic and canonical and thus
    $(\sigma_{U_{\phi}})^a=\sigma_{U_{\phi}}$.
    The claim, \eqref{equ:canon_inv}, follows.
\end{proof}
This finishes the proofs of Theorem~\ref{t:oddabelianschur}
and Proposition~\ref{p:oddabschur_sup}.

\section[Isaacs correspondence]{Isaacs correspondence and Schur indices}\label{sec:appl}
\begin{thm}\label{t:strongschur}
  Let $L, K\nteq G$ with $L\leq K$ and $K/L$ of odd order.
  Suppose there is $M\leq G$ with
  $MK\nteq G$, $(\abs{M/L}, \abs{K/L})=1$ and
  $\C_{K/L}(M)=1$.
  Let $\phi\in \Irr_M L$ and $\theta\in \Irr_M K$ with
  $(\theta_L, \phi)>0$ and set $H=\N_G(M)$.
  Set $e= (e_{(\theta, \rats)})^G$ and
  $f = (e_{(\phi, \rats)})^H$.
  Then
  $\rats G e \iso \mat_n(\rats H f)$, where
  $n=\theta(1)/\phi(1)$.
  There is a canonical bijection between
  $\Irr(G\mid e)$ and $\Irr(H\mid f)$ which commutes with
  field automorphisms and preserves Schur indices.
  (More precisely, the correspondence has
   Properties~\ref{i:c_irr}--\ref{i:c_brauer}
   of Proposition~\ref{p:oddabschur_sup}.)
\end{thm}
We need the following version of the
going-down
theorem~\cite[Theorem~6.18]{isaCTdov} for semi-invariant characters:
\begin{prop}\label{p:gdsemi}
  Let $K/L$ be an abelian chief factor of $G$ and
  suppose $\theta\in \Irr K$ is\/ $\crp{F}$-semi-invariant in $G$
  for some field\/ $\crp{F}\subseteq \compl$.
  Then one of the following holds:
  \begin{enums}
  \item \label{i:gd_ind} $\theta = \phi^K$ with $\phi\in \Irr L$,
         and either
         \begin{enums}
         \item \label{ii:gd_ind1} $\crp{F}(\phi) = \crp{F}(\theta)$, or
         \item \label{ii:gd_ind2} $K/L \iso \Gal( \crp{F}(\phi)/\crp{F}(\theta))$
               and
              $\phi$ is\/ $\crp{F}$-semi-invariant in $G$.
         \end{enums}
  \item \label{i:gd_res} $\theta_L =\phi \in \Irr L$.
  \item \label{i:gd_fr} $\theta_L = e\phi$ with $\phi\in \Irr L $ and
         $e^2= \abs{K/L}$, and\/ $\crp{F}(\theta)= \crp{F}(\phi)$.
  \end{enums}
\end{prop}
\begin{proof}
  Let $\phi$ be an irreducible constituent of $\theta_L$.
  Let
  \[ T = \{ g\in G\mid \phi^g \text{ is Galois conjugate to } \phi
                               \text{ over } \crp{F}\}.\]
  Let $g\in G$ and pick $\alpha\in \Gal(\crp{F}(\theta)/\crp{F})$
  with $\theta^g = \theta^{\alpha}$.
  Then
  \[ ( (\theta^{\alpha})_L, \phi^{\alpha})
       = ( \theta_L, \phi)
       = (\theta_L^g, \phi^g)
       = (\theta^{\alpha}_L, \phi^g),\]
  and thus $\phi^{\alpha} = \phi^{gk}$ for some $k\in K$.
  It follows that $G= KT$.
  Since $K/L$ is abelian, $K\cap T \nteq KT=G$ and thus either
  $K\cap T = L$ or $K\cap T = K$.

  If $K\cap T = L$, then $\phi^K = \theta$.
  Thus clearly $\crp{F}(\theta) \leq \crp{F}(\phi)$.
  Let $\alpha\in \Gal(\crp{F}(\phi)/\crp{F}(\theta))$. Then
  $(\theta_L, \phi^{\alpha}) = (\theta_L, \phi)>0$ and thus
  $\phi^{\alpha} = \phi^k$ for some $k\in K\cap T = L$, so that
  $\phi^{\alpha}= \phi$. It follows that
  $\Gal( \crp{F}(\phi)/\crp{F}(\theta))= 1$, and thus
  $\crp{F}(\theta)= \crp{F}(\phi)$, which is possibility~\ref{ii:gd_ind1}
  in situation~\ref{i:gd_ind}.

  Now suppose $K\cap T = K$, that is $T=G$ and $\phi$ is
  semi-invariant in $G$.
  Consider the inertia group $K_{\phi}$.
  For $g\in G$, there is $\alpha_g\in \Aut \crp{F}(\phi)$
  with $\phi^g = \phi^{\alpha_g}$, so that
  \[  K_{\phi}^g = K_{\phi^g} = K_{\phi^{\alpha_g}}= K_{\phi}.\]
  It follows that $K_{\phi}\nteq G$.
  Again, either
  $K_{\phi}= K$ or $K_{\phi}= L$.

  If $K_{\phi}= L$, then again
  $\phi^K = \theta\in \Irr K$,
  but now $\phi$ is semi-invariant in $K$.
  Since $\theta^k=\theta$ for $k\in K$, the homomorphism
  of Lemma~\ref{l:fnsg} maps $K/L$ into the Galois group
  $\Gal(\crp{F}(\phi)/\crp{F}(\theta) )$.
  Conversely, for $\alpha\in \Gal(\crp{F}(\phi)/\crp{F}(\theta))$
  we have
  $(\theta,\phi^{\alpha})=(\theta^{\alpha},\phi^{\alpha})=1$ and
  thus $\phi^{\alpha}$ and $\phi$ are conjugate in $K$.
  It follows that the homomorphism of Lemma~\ref{l:fnsg} is onto,
  and thus
  $K/L \iso \Gal(\crp{F}(\phi)/\crp{F}(\theta))$.
  This is situation~\ref{i:gd_ind}\ref{ii:gd_ind2}.

  Now assume $K_{\phi}=K$, so that $\phi$ is invariant in $K$.
  Set
  \[ \Lambda = \{ \lambda\in \Lin(K/L) \mid \theta\lambda = \theta\}
     \quad \text{and} \quad
     U = \bigcap_{\lambda \in \Lambda} \Ker \lambda.\]
  We claim that $U\nteq G$.
  If $\theta\lambda= \theta$, then
  $\theta^{\alpha}\lambda = \theta^{\alpha}$ for field
  automorphisms $\alpha$, as $\theta^{\alpha}$ and $\theta$ have the same
  zeros.
  Let $g\in G$ and $\lambda \in \Lambda$.
  From the semi-invariance of $\theta$ it follows that
  there is $\alpha\in \Aut \crp{F}(\theta)$
  with $\theta^{\alpha g}=\theta$. Thus
  \[ \theta \lambda^g
     = \theta^{\alpha g} \lambda^g
     = (\theta^{\alpha}\lambda)^g
     = \theta^{\alpha g} = \theta.
  \]
  Thus $\Lambda$ is invariant in $G$, and it follows that
  $U\nteq G$.
  Thus either $U=K$ or $U=L$.

  If $U=K$, then $\Lambda=\{1\}$ and thus
  the $\theta\lambda$ with $\lambda\in \Lin(K/L)$ are $\abs{K/L}$
  different constituents of $\phi^K$ occurring with the same
  multiplicity, $e$, so that
  \[ \abs{K/L}\phi(1)= \phi^K(1)
         = e \abs{K/L}\theta(1) = e^2\abs{K/L}\phi(1),\]
  and it follows $e=1$ (situation~\ref{i:gd_res}).

  If $U=L$, then $\theta$ vanishes on $K\setminus
  L$, and thus $\phi$ is fully ramified in $K$
  (situation~\ref{i:gd_fr}).
  It is clear that then $\crp{F}(\theta)=\crp{F}(\phi)$.
\end{proof}
In situation~\ref{i:gd_res}, we clearly have
$\crp{F}(\phi)\leq \crp{F}(\theta)$, and
\[\{\theta^{\alpha}
    \mid \alpha\in \Gal( \crp{F}(\theta)/\crp{F}(\phi) )
  \}
  \subseteq \Irr(K\mid \phi)
  = \{\theta \lambda \mid \lambda\in \Lin(K/L) \}.
\]
Thus
$1
 \leq \abs{\crp{F}(\theta):\crp{F}(\phi)}
 \leq \abs{K/L}$.
In our intended application, we will have
$\crp{F}(\theta)= \crp{F}(\phi)$.
Then the following result, probably well known, will be useful.
\begin{prop}\label{p:rescorr}
  Let $H\leq G$ be finite groups, $K\nteq G$ with $G=HK$,
   and set $L= H\cap K$.
  Assume that $\theta\in \Irr K$ is semi-invariant in $G$, that
  $\theta_L=\phi \in \Irr L$ and that
  $\crp{F}(\theta)=\crp{F}(\phi)$ for some field
  $\crp{F}\subseteq\compl$.
  Set $e = e_{(\theta, \crp{F})}$ and
  $ f = e_{(\phi, \crp{F})}$.
  Then
  \[ \crp{F}Hf \ni a\mapsto ae \in \crp{F}G e\]
  is an isomorphism of $H/L$-graded algebras.
\end{prop}
\begin{cor}
  In the situation of the proposition, restriction defines a
  bijection $\Irr(G\mid e)\to \Irr(H\mid f)$ commuting with field
  automorphism over $\crp{F}$ and preserving Schur indices  over
  $\crp{F}$. (More precisely, the correspondence has
   Properties~\ref{i:c_irr}--\ref{i:c_brauer}
   of Proposition~\ref{p:oddabschur_sup},
   with $\rats$ replaced by $\crp{F}$.)
\end{cor}
\begin{proof}
  The isomorphism $\crp{F}Hf \to \crp{F}Ge$ defines a map from
  $\Irr(G\mid e)$ to $\Irr(H\mid f)$, sending
  $\chi\in \Irr(G\mid e)$ to
  ${\chi}^{o}$ with
  ${\chi}^{o}(h) = \chi( he)= \chi(h)$.
  It is clear that this commutes with field automorphisms. The
  part on the Schur indices follows since
  $\crp{F}H e_{(\chi^{o},\crp{F})}
    \iso \crp{F}G e_{(\chi,\crp{F})}$.
\end{proof}
\begin{proof}[Proof of Proposition~\ref{p:rescorr}]
  Let $V$ be an absolutely irreducible module affording $\theta$.
  Then $Ve_{\phi}=V$ and $Ve_{\tilde{\phi}}=0$ for any other
  $\tilde{\phi}\in \Irr(L)$. It follows that
  $e_{\phi}e_{\theta}=e_{\theta}$ and
  $e_{\tilde{{\phi}}}e_{\theta}=0$ for $\tilde{\phi}\neq \phi$.
  In particular, this holds for
  $\tilde{\phi}=\phi^{\alpha}$ when
  $1\neq \alpha\in \Gal(\crp{F}(\phi)/\crp{F})$.
  Since we assume that $\crp{F}(\phi)=\crp{F}(\theta)$, it follows
  that
  $e_{\phi^{\alpha}} e_{\theta^{\beta}}
    = \delta_{\alpha,\beta} e_{\theta^{\beta}}$ for
  $\alpha$, $\beta\in \Gal(\crp{F}(\phi)/\crp{F})$.
  Thus $ef = e$.

  Now $a\mapsto ae$ maps $\crp{F}Lf$ into $\crp{F}Ke$.
  Since $\crp{F}Lf$ is simple, the map is injective, and since
  $\crp{F}Lf$ and $\crp{F}Ke$ both have dimension
  $\phi(1)=\theta(1)$ over its center, the map is an
  isomorphism.
  Finally, for $h\in H$, we get
  $\crp{F}Lf h \cdot e = \crp{F}K e h$.
  The proof follows.
\end{proof}
We also need a standard fact about
coprime action~\cite[Theorems~13.27, 13.28, 13.31 and Problem~13.10]{isaCTdov}
or~\cite[Corollary~2.4 and Lemma~2.5]{i82}.
\begin{lemma}\label{l:cpact}
  Let $A$ act on $K$ and let $L\nteq K$ be $A$-invariant.
  Suppose $(\abs{A}, \abs{K/L})=1$ and
  $\C_{K/L}(A)=1$. Then
  \begin{enums}
  \item If $\theta\in \Irr_A K$ then $\theta_L$ has a unique
        $A$-invariant constituent.
  \item If $\phi\in \Irr_A L$, then $\phi^K$ has a unique
         $A$-invariant constituent.
  \end{enums}
\end{lemma}
The proof of the second assertion is relatively elementary if
$K/L$ is abelian~\cite[2.5]{i82} and can be reduced to that case
if $K/L$ is solvable. We will only need this case.
(In the case where $K/L$ is not solvable, the proof
 depends on the Glauberman correspondence.)
The first assertion is easy in any case.
\begin{proof}[Proof of Theorem~\ref{t:strongschur}]
  Suppose $G$ is a counterexample with $\abs{G/L}$ minimal.

  As $M/L$ acts coprimely and fixed point freely on $K/L$, it
  follows that above every $\phi\in \Irr_M L$,
  there lies a unique $\theta\in
  \Irr_M K$, and conversely (see Lemma~\ref{l:cpact}).
  Since this bijection is natural, it  commutes with the action of $H$ and
  with Galois action.
  In particular, $\rats(\theta)= \rats(\phi)$ and
  $H_{\theta}= H_{\phi}$.

  Set $e_1= e_{(\theta, \rats)}$
  and $f_1= e_{(\phi, \rats)}$.
  Let $U$ be the stabilizer of $e_1$ in
  $G$. Then $V=U\cap H$ is the stabilizer of $f_1$ in $H$.
  By Proposition~\ref{p:cliffordcorr} we have
  $\rats Ge \iso \mat_{\abs{G:U}}(\rats U e_1)$ and
  $\rats Hf \iso \mat_{\abs{H:V}}(\rats V f_1)$,
  and canonical character correspondences.
    If $U<G$, then induction applies and yields an isomorphism
  $\rats U e_1 \iso \mat_n( \rats V f_1)$ as in the theorem and a
  canonical character correspondence.
  This yields
  \[ \rats G e \iso \mat_{\abs{G:U}}\big( \mat_n( \rats V f_1)\big)
               \iso \mat_n \big( \mat_{\abs{H:V}}( \rats V f_1) \big)
               \iso \mat_n \big( \rats H f \big)\]
  and canonical character correspondences.
  Thus the configuration can not be a minimal counterexample.
  It follows that $U=G$, that is, $\theta$ is semi-invariant in
  $G$.

  In a counterexample, we must have $L<K$.
  Let $L<N\leq K$ with $N/L$ a chief factor.
  There is a unique $\eta\in \Irr_M N$ that lies above $\phi$, and
  this $\eta$ is a constituent of $\theta_N$ (Lemma~\ref{l:cpact}).
  This $\eta$ is also semi-invariant in $H$ and has the same field of
  values as $\phi$ and $\theta$.
  Let $U=\N_G(MN)$. If $N< K$, then induction applies to yield
  isomorphisms
  $\rats Ge \iso \mat_{n_1}(\rats U i)$ and
  $\rats U i \iso \mat_{n_2}(\rats H f)$
  (where $i= e_{(\eta, \rats)}$),
  and
  natural bijections between $\Irr(G\mid \theta)$ and $\Irr(U\mid \eta)$
  with the required properties,
  and between $\Irr(U\mid \eta)$ and $\Irr(H\mid \phi)$.
  It follows that
  $\rats G e \iso \mat_{n_1 n_2}(\rats H f)$ with
  $n_1n_2 = (\theta(1)/\eta(1))\cdot (\eta(1)/\phi(1))
          = \theta(1)/\phi(1)$, and that there is a natural
  bijection between
  $\Irr(G \mid \theta)$ and $\Irr(H\mid \phi)$.

  We may thus assume that $K/L$ is a chief factor of $G$.
  Then, according to the ``going down'' result for
  semi-\hspace{0pt}invariant
  characters (Proposition~\ref{p:gdsemi}),
  one of three possibilities occurs.

  First, suppose we are in Situation~\ref{i:gd_ind}
  of Proposition~\ref{p:gdsemi}, so that $\phi^K = \theta$.
  Here $\phi$ can not be
  semi-invariant in $K$, since this would imply
  $\rats(\theta)< \rats(\phi)$ which is impossible.
  It follows that $H$ is the inertia group of $f$.
  Then  Proposition~\ref{p:cliffordcorr} applies and yields the
  result.
  (In this case, $ n = \abs{K/L}$.)

  Now suppose that $\theta_L=\phi\in \Irr L$.
  As $\rats (\theta)=\rats( \phi)$, Proposition~\ref{p:rescorr}
  applies  and yields the result.
  (In this case, $n=1$. )

  Thus we assume that $\phi$ is fully ramified in $K$.
  Then Theorem~\ref{t:oddabelianschur} and Proposition~\ref{p:oddabschur_sup}
  (with $N=MK$) apply and yield the result.
  (In this case, $n = \sqrt{\abs{K/L}}$.)
\end{proof}
Note that oddness of $\abs{K/L}$ was only applied in the last
sentence of the proof (if solvability is assumed).
Nevertheless the result is false for $\abs{K/L}$ even.

Now assume that $N$ is a finite group on which the group $A$ acts.
Suppose that $\abs{N}$ and $\abs{A}$ are relatively prime,
and that $\abs{N}$ is odd.
As mentioned in the introduction, Isaacs used his results on fully
ramified sections to construct a correspondence between
$\Irr_A N$ and $\Irr \C_N(A)$.
We call this the Isaacs correspondence.
(Strictly speaking, we should call  it the Isaacs part of the
 Glauberman-Isaacs correspondence.)
We will need to recall the construction of the Isaacs
correspondence in the proof of the next result.
\begin{cor}\label{c:isaacscor}
  Let $N$ be a finite group of odd order,
  $A$ a group such that $\abs{N}$ and $\abs{A}$ are relatively
  prime, and suppose the semidirect product $AN$ is a normal
  subgroup of a finite group $G$.
  Set $C=\C_{N}(A)$ and $U=\N_G(A)$.
  Let $\chi\in \Irr_{A}(N)$ and
  $\chi^{*}\in \Irr C$ be its Isaacs correspondent.
  Set $i= (e_{ (\chi, \rats ) })^G$ and
  $i^{*} = (e_{ (\chi^{*}, \rats) })^U$.
  Then
  $\rats G i \iso \mat_n(\rats U i^{*})$
  as $(U/C)$\nbd graded algebras,
  with
  $n=\chi(1)/\chi^{*}(1)$.
  There is a natural correspondence between
  $\Irr(G\mid i)$ and $\Irr(U\mid i^{*})$
  preserving Schur indices.
\end{cor}
\begin{cor}
  The Isaacs correspondence preserves Schur indices.
\end{cor}
\begin{proof}
  The isomorphism of Corollary~\ref{c:isaacscor} restricts to an
  isomorphism $\rats N i\iso \mat_n(\rats C i^{*})$.
\end{proof}
In the situation of Corollary~\ref{c:isaacscor},
observe that $ G= N U$ by the Frattini argument, and that
$C= N\cap U$ since $(\abs{A}, \abs{N})=1$.
It follows that
$G/N \iso U/C$, and it makes sense to compare the character
sets above $\chi$ respective $\chi^{*}$.
G.~Navarro~\cite{nav94} attributes to L.~Puig the question if the
Clifford extensions of $G/N$ and $U/C$ associated to
$\chi$ and $\chi^{*}$ are
isomorphic in this case.
This has been answered in the affirmative by
M.~L.~Lewis~\cite{lewis97b}.
(To be exactly, he shows that the associated character triples
 are isomorpic, which is somewhat weaker.)
Corollary~\ref{c:isaacscor} generalizes this result.
If $\abs{N}$ even, that is, we are in the situation of the
Glauberman correspondence, the result is false. However, it is
true if we work over $\compl$ instead of
$\rats$~\cite{dade80, puig86, lewis97b}.
If $A$ is a $p$\nbd group, then it is true over $\rats_p$, the
$p$\nbd adic numbers~\cite{turull08c}.
\begin{proof}[Proof of Corollary~\ref{c:isaacscor}]
  Let $K=[N,A]$ and $L= K'$.
  Then $L$, $K\nteq G$.
  We may assume that $K>1$ (otherwise $C=N$ and $G=U$).
  It follows that $L<K$.
  By results on coprime action, $\C_{K/L}(A)=1$.
  \begin{figure}
  \setlength{\unitlength}{0.3ex}
  \centering
  \begin{picture}(120,132)(0,0)
    \put(3,63){\line(1,1){14}}  
    \put(23,83){\line(1,1){44}}  
    \put(33,33){\line(1,1){14}}  
    \put(53,53){\line(1,1){44}}  
    \put(73,33){\line(1,1){44}}  
    \put(3,57){\line(1,-1){24}}  
    \put(23,77){\line(1,-1){24}}  
    \put(53,47){\line(1,-1){14}}  
    \put(73,127){\line(1,-1){24}}  
    \put(103,97){\line(1,-1){14}}  
    \put(73,27){\line(1,-1){24}}   
    \put(103,3){\line(1,1){14}}    
    \put(-4,58){{$K$}}
    \put(-4,49){{$\theta$}}
    \put(16,77){$N$}
    \put(16,69){{$\chi$}}
    \put(67,126){$G$}
    \put(27,26){{$L$}}
    \put(27,18){{$\phi$}}
    \put(97,98){$H$}
    \put(45,47){{$LC$}}
    \put(46,39){{$\psi$}}
    \put(67,27){$C$}
    \put(66,19){$\chi^{*}$}
    \put(117,78){$U$}
    \put(117,17){$A$}
    \put(98,-3){$1$}
  \end{picture}
  \caption{Corollary~\ref{c:isaacscor}}
  \label{fig:aboveisaacs_p}
\end{figure}

  There is an $A$-invariant constituent
  $\theta$ of $\chi_K$~\cite[Theorem 13.27]{isaCTdov}.
  By Lemma~\ref{l:cpact}, there is
  a unique $A$-invariant constituent $\phi$ of $\theta_L$.
  Set $M=AL$ and $H=\N_G(M)$.
  Then $MK=AK\nteq G$, $\C_{K/L}(M)=1$ and, by the Frattini
  argument,
  $H=\N_{G}(A)L =UL$ (see Figure~\ref{fig:aboveisaacs_p}).

  Now Theorem~\ref{t:strongschur} applies and yields an isomorphism
  $\kappa\colon \rats G e \to \mat_{n_1}(\rats H f)$, where $e$ and $f$ are as
  in Theorem~\ref{t:strongschur}. The natural correspondence sends
  $\chi$ to an character $\psi\in \Irr(LC)$.
  By the inductive definition of the Isaacs
  correspondence, $\psi$ is the Isaacs
  correspondent of $\chi^{*}$.
  Let $j$ be the idempotent in $\Z(\rats LC)^H$ belonging to $\psi$.
  Then $\kappa$ restricts to an isomorphism
  $\rats G i \to \mat_{n_1}(\rats H j)$.
  Observe that $n_1= \chi(1)/\psi(1)$.
  By induction, $\rats H j \iso \mat_{n_2}(\rats U i^{*})$ with
  $n_2 = \psi(1)/\chi^{*}(1)$.
  The result now follows.
\end{proof}
%
\printbibliography
\end{document}